\documentclass[letterpaper,11pt]{article}
\usepackage[margin=1.25in]{geometry}
\usepackage[bookmarks, colorlinks=true, plainpages = false, citecolor = blue,linkcolor=red,urlcolor = blue, filecolor = blue]{hyperref}
\usepackage{url}
\usepackage{amsmath,amsfonts,amsthm,amssymb,bm, verbatim,dsfont,mathtools}
\usepackage{algorithm,algorithmic,color,graphicx,appendix}
\usepackage{amsmath,amsfonts,amscd,amssymb,bm,epsfig,epsf,bbm,constants,mathabx}
\usepackage{cases}
\usepackage{subfigure}
\usepackage{float}
\usepackage{etoolbox}
\usepackage{tikz}
\usepackage{xr,xspace}
\usepackage{todonotes}
\usepackage{paralist}
\usepackage{caption}
\usepackage{prettyref}
\usepackage[square]{natbib}
\newrefformat{alg}{Algorithm~\ref{#1}}
\theoremstyle{plain}
\newtheorem{theorem}{Theorem}
\newtheorem{lemma}{Lemma}

\newtheorem{corollary}{Corollary}
\theoremstyle{definition}
\newtheorem{definition}{Definition}

\newtheorem*{remark*}{Remark}

\numberwithin{equation}{section}
\usepackage{xspace,prettyref}

\newcommand{\reals}{{\mathbb{R}}}

\newcommand{\Expect}{\mathbb{E}}
\newcommand{\expect}[1]{\mathbb{E}\left[ #1 \right]}

\newcommand{\prob}[1]{{ \mathbb{P}\left\{ #1 \right\} }}

\newcommand{\var}{\mathsf{var}}

\newcommand{\Tr}{{\rm Tr}}
\newcommand{\eg}{e.g.\xspace}
\newcommand{\ie}{i.e.\xspace}

\newrefformat{eq}{(\ref{#1})}
\newrefformat{chap}{Chapter~\ref{#1}}
\newrefformat{sec}{Section~\ref{#1}}
\newrefformat{algo}{Algorithm~\ref{#1}}
\newrefformat{fig}{Fig.~\ref{#1}}
\newrefformat{tab}{Table~\ref{#1}}
\newrefformat{rmk}{Remark~\ref{#1}}
\newrefformat{clm}{Claim~\ref{#1}}
\newrefformat{def}{Definition~\ref{#1}}
\newrefformat{cor}{Corollary~\ref{#1}}
\newrefformat{lmm}{Lemma~\ref{#1}}
\newrefformat{prop}{Proposition~\ref{#1}}
\newrefformat{app}{Appendix~\ref{#1}}
\newrefformat{hyp}{Hypothesis~\ref{#1}}
\newrefformat{thm}{Theorem~\ref{#1}}

\newcommand{\iprod}[2]{\left \langle #1, #2 \right\rangle}
\newcommand{\Iprod}[2]{\langle #1, #2 \rangle}
\newcommand{\indc}[1]{{\mathbf{1}_{\left\{{#1}\right\}}}}

\newcommand{\calE}{{\mathcal{E}}}
\newcommand{\calF}{{\mathcal{F}}}

\newcommand{\diag}{{\rm diag}}

\definecolor{darkred}{RGB}{100,0,0}
\definecolor{darkgreen}{RGB}{0,100,0}
\definecolor{darkblue}{RGB}{0,0,150}
\definecolor{red}{RGB}{255,0,0}

\newcommand{\Var}{\textrm{Var}}

\renewcommand{\P}{\operatorname{\mathbb{P}}}
\newcommand{\E}{\operatorname{\mathbb{E}}}

\renewcommand{\vec}[1]{{\boldsymbol{#1}}}

\newcommand{\vct}[1]{\bm{#1}}
\newcommand{\mtx}[1]{\bm{#1}}

\newcommand{\rank}{\operatorname{rank}}

\renewcommand{\hat}{\widehat}

\newcommand{\SBM}{{\sf SBM}\xspace }
\newcommand{\DCSBM}{{\sf DCSBM}\xspace }
\newcommand{\Pnr}{\ensuremath{ \mathbb{P}_{n, r} }}
\newcommand{\Mnr}{\ensuremath{ \mathbb{M}_{n, r} }}
\newcommand{\Yhat}{\ensuremath{ \widehat{\mtx{Y}} }}
\newcommand{\Ystar}{\ensuremath{ \mtx{Y}^* }}
\newcommand{\PsiCheck}{\ensuremath{ \widecheck{\mtx{\Psi}} }}
\newcommand{\XCheck}{\ensuremath{ \widecheck{\mtx{X}} }}

\newcommand{\silent}[1]{}

\definecolor{xl}{RGB}{200,50,120}
\definecolor{yc}{RGB}{50,150,50}
\definecolor{jx}{RGB}{50,50,200}

\begin{document}

\title{Convexified Modularity Maximization for Degree-corrected Stochastic Block Models}

\date{\today}
\author{ 
Yudong Chen\thanks{Y. Chen is
with the School of Operations Research and Information Engineering at Cornell University, Ithaca, NY. \texttt{yudong.chen@cornell.edu}.}
  \and Xiaodong Li\thanks{X. Li is with the Statistics Department at the University of California, Davis, CA.
\texttt{xdgli@ucdavis.edu}.}
   \and Jiaming Xu\thanks{J. Xu is with Department of Statistics, The Wharton School, University of Pennsylvania, Philadelphia, PA. \texttt{jiamingx@wharton.upenn.edu}.}
}

\maketitle

\begin{abstract}
The stochastic block model (\SBM) is a popular framework for studying community detection in networks.
This model is limited by the assumption that all nodes in the same community are statistically equivalent and have equal expected degrees.
The degree-corrected stochastic block model (\DCSBM) is a natural extension of \SBM that allows for degree heterogeneity within communities.
This paper proposes a \emph{convexified modularity maximization} approach for estimating the hidden communities under \DCSBM. Our approach is based on a convex programming relaxation of the classical (generalized) modularity maximization formulation, followed by a novel doubly-weighted $ \ell_1 $-norm $ k $-median procedure.
We establish non-asymptotic theoretical guarantees for both approximate clustering and perfect clustering.  Our approximate clustering results are  insensitive to the minimum degree, and hold even in sparse regime with bounded average degrees. In the special case of \SBM, these theoretical results match the best-known performance guarantees of computationally feasible algorithms. 
Numerically, we provide an efficient implementation of our algorithm, which is applied to both synthetic and real-world networks. Experiment results show that our method enjoys competitive performance compared to the state of the art in the literature.
\end{abstract}


\section{Introduction}

\label{SecIntro}

\label{sec:introduction}

Detecting communities/clusters in networks and graphs is an important subroutine in many applications across computer, social and natural sciences and engineering. A standard framework for studying community detection in a statistical setting is the stochastic block model (\SBM) proposed in~\cite{HLL1983}. Also known as the planted partition model in the computer science literature~\citep{Cordon01}, \SBM is a random graph model for generating networks from a set of underlying clusters. The statistical task is to accurately recover the underlying true clusters given a single realization of the random graph.

The versatility and analytic tractability of  \SBM have made it arguably the most popular model for studying community detections. It however falls short of abstracting a key aspect of real-world networks. In particular, an unrealistic assumption of \SBM
is that within each community, the degree distributions of each node are the same. In empirical network data sets, however, the degree distributions are often highly inhomogeneous across nodes, sometimes exhibiting a heavy tail behavior with some nodes having very high degrees (so-called hubs). At the same time, sparsely connected nodes with small degrees are also common in real networks. To overcome this shortcoming of the \SBM, the degree-corrected stochastic block model (\DCSBM) was introduced in the literature to allow for degree heterogeneity within communities, thereby providing a more flexible and accurate model of real-world networks \citep{DHM2004, KN2011}.

A number of community detection methods have been proposed based on \DCSBM, such as model-based methods and spectral methods.
Model-based methods include profile likelihood maximization and modularity maximization \citep{Newman2006modularity,KN2011}. Although these methods enjoy certain statistical guarantees \citep{ZLZ2012}, they often involve optimization over all possible partitions, which is computationally intractable. Recent work in \cite{ACBL2013,LLV14} discusses efficient solvers, but the theoretical guarantees are only established under restricted settings such as those with two communities. Spectral methods, which estimate the communities based on the eigenvectors of the graph adjacency matrix and its variants, are computationally fast. Statistical guarantees are derived for spectral methods under certain settings (see, \eg, \cite{DHM2004,CL2009,chaudhuri,qin2013regularized,LR2013,Jin2012,GLM15}), but numerical validation on synthetic and real data has not been as thorough.
One notable exception is the SCORE method proposed in \cite{Jin2012}, which achieved one of the best known clustering performance on the political blogs dataset from~\cite{AG2005}. Spectral methods are also known to suffer from inconsistency in sparse graphs \citep{KMMNSZZ2013} as well as sensitivity to outliers \citep{Cai2014robust}. We discuss other related work in details in \prettyref{sec:related}.

In this paper, we seek for a clustering algorithm that is computationally feasible, has strong statistical performance guarantees under \DCSBM, and provides competitive empirical performance.
Our approach makes use of the robustness and computational power of convex optimization. Under the standard \SBM, convex optimization has been proven to be statistically efficient under a broad range of model parameters, including the size and number of communities as well as the sparsity of the network; see e.g.\ \cite{CSX2013, ChenXu14, Vershynin14,AV2011,OH2011}. Moreover, a significant advantage of convex methods is their robustness against arbitrary outlier nodes, as is established in the theoretical framework in \cite{Cai2014robust}. There, it is also observed that their convex optimization approach leads to state-of-the-art misclassification rates in the political blogs dataset, in which the node degrees are highly heterogeneous. These observations motivate us to study whether strong theoretical guarantees under \DCSBM can be established for convex optimization-based methods.

Building on the work of~\cite{CSX2013} and~\cite{Cai2014robust}, we introduce in \prettyref{sec:methodology} a new community detection approach called Convexified Modularity Maximization (CMM). CMM is based on convexification of the elegant modularity maximization formulation, followed by a novel and computationally tractable weighted $ \ell_1 $-norm $ k $-median clustering procedure. As we show in~\prettyref{sec:theory} and~\prettyref{sec:simulation}, our approach has strong theoretical guarantees, applicable even in the sparse graph regime with bounded average degree, as well as state-of-the-art empirical performance. In both aspects our approach is comparable to or improves upon the best-known results in the literature.

\section{Problem setup and algorithms}
\label{sec:methodology}

In this section, we set up the community detection problem under \DCSBM, and describe our algorithms based on convexified modularity maximization and weighted $k$-median clustering. 

Throughout this paper, we use lower-case and upper-case bold letters such as $ \vct{u}$ and $ \mtx{U} $ to represent vectors and matrices, respectively, with $ u_i $ and $ U_{ij} $ denoting their elements. We use $ \mtx{U}_{i*} $  to denote the $ i $-th row of the matrix $ \mtx{U} $.
If all coordinates of a vector $\vct{v}$ are nonnegative, we write $\vct{v}\geq \vct{0}$. The notation $\vct{v}>\vct{0}$, as well as $\mtx{U} \geq 0$ and $\mtx{U}>0$ for matrices, are defined similarly.
For a symmetric matrix $\mtx{U} \in \mathbb{R}^{n \times n}$, we write $\mtx{U}\succ \mtx{0}$ if $ \mtx{U} $ is positive definite, and $\mtx{U}\succeq \mtx{0}$ if it is positive semidefinite. 
For any sequences $\{a_n\}$ and $\{b_n\}$,  we write $a_n  \lesssim b_n$ 
if there is an absolute constant $c>0$ such that $a_n/ b_n \le c, \forall n$, and we define $a_n \gtrsim b_n$ similarly. 

\subsection{The degree-corrected stochastic block model}
\label{sec:model}

In \DCSBM a graph $ \mathcal{G} $ is generated randomly as follows. A total of $n$ nodes, which we identify with the set $ [n] := \{1, \ldots, n\}$, are partitioned into $r$ fixed but unknown clusters $C^\ast_1, C^\ast_2 \ldots, C^\ast_r$. Each pair of distinct nodes $i \in C^*_a$ and $j \in C^*_b$ are connected by an (undirected) edge with probability $\theta_i \theta_j B_{ab} \in [0,1]$, independently of all others. Here the vector $\vct{\theta}=(\theta_1, \ldots, \theta_n)^\top \in \reals^{n}_+ $ is referred to as the \emph{degree heterogeneity parameters} of the nodes, and the symmetric matrix $\mtx{B} \in \reals^{r\times r}_+$ is called the \emph{connectivity matrix} of the clusters. Without loss of generality, we assume that $ \max_{1\le i \le n} \theta_i =1 $ (one can multiply $\vct{\theta}$ by a scalar $c$ and divide $\mtx{B}$ by $c^2$ without changing the distribution of the graph).
Note that if $\theta_i=1$ for all nodes~$i$, DCSBM reduces to the classical \SBM. Given a single realization of the resulting random graph $\mathcal{G} = ([n], E)$, the statistical goal is to estimate the true clusters $\{C^\ast_a\}_{a=1}^r$.

Before describing our algorithms,  let us first introduce some useful notation. Denote by $\mtx{A} \in \{0, 1\}^{n \times n}$ the \emph{adjacency matrix}  associated with the graph $ \mathcal{G} $, with $A_{ij}=1$ if and only if nodes $i$ and $j$ are connected. For each candidate partition of $n$ nodes into $r$ clusters, we associate it with a \emph{partition matrix} $\mtx{Y} \in \{0, 1\}^{n \times n}$, such that $Y_{ij}=1$ if and only if nodes $i$ and $j$ are assigned to the same cluster, with the convention that $ Y_{ii} = 1, \forall i $. Let $\Pnr$ be the set of all such partition matrices, and $ \Ystar $ the true partition matrix associated with the ground-truth clusters $\{C^\ast_a\}_{a=1}^r$. The notion of partition matrices plays a crucial role in the subsequent discussion.

\subsection{Generalized modularity maximization}

Our clustering algorithm is based on Newman and Girvan's classical notion of  \emph{modularity} (see, e.g., \cite{Newman2006modularity}). Given the graph adjacency matrix $\mtx{A}$ of $ n $ nodes, the modularity of a partition represented by the partition matrix $\mtx{Y} \in \bigcup_r \Pnr$, is defined as
\begin{equation}
\label{eq:modularity}
Q(\mtx{Y}) := \sum_{1 \leq i, j\leq n}\left(A_{ij}-\frac{d_{i}d_{j}}{2L}\right)Y_{ij},
\end{equation}
where $d_{i}:=\sum_{j=1}^n A_{ij}$ is the degree of node $i$, and $L= \frac{1}{2}\sum_{i=1}^n d_{i}$
is the total number of edges in the graph. The \emph{modularity maximization} approach to community detection is based on finding a partition $\mtx{Y}_{m}$ that optimizes $ Q(\mtx{Y}) $:
\begin{equation}
\label{eq:modularity1}
\mtx{Y}_{m}\leftarrow \arg\max_{\mtx{Y} \in \bigcup_r \Pnr} Q(\mtx{Y}).
\end{equation}
This standard form of modularity maximization is known to suffer from a ``resolution limit'' and cannot detect small clusters~\citep{fortunato2007resolution}. To address this issue, several authors have proposed to replace the normalization factor $\frac{1}{2L}$ by a tuning parameter~$ \lambda $ \citep{RB2006,LF2011}, giving rise to the following generalized formulation of modularity maximization:
\begin{equation}
\label{eq:modularity2}
\mtx{Y}_{m}\leftarrow \arg\max_{\mtx{Y} \in \bigcup_r \Pnr} Q_\lambda (\mtx{Y}) := \sum_{1 \leq i, j\leq n}\left(A_{ij}-\lambda d_{i}d_{j}\right)Y_{ij}.
\end{equation}

While modularity maximization enjoys several desirable statistical properties under \SBM and \DCSBM \citep{ZLZ2012, ACBL2013}, the associated optimization problems~\prettyref{eq:modularity1} and~\prettyref{eq:modularity2} are not computationally feasible due to the combinatorial constraint, which limits the practical applications of these formulations. In practice, modularity maximization is often used as a guidance for designing heuristic algorithms.

Here we take a more principled approach to computational feasibility while maintaining provable statistical guarantees:  we develop a tractable convex surrogate for the above combinatorial optimization problems, whose solution is then refined by a novel weighted $ k $-median algorithm.

\subsection{Convex relaxation}
\label{sec:cvx}

Introducing the degree vector $\vct{d}=(d_1, \ldots, d_n)^\top$, we can rewrite the generalized modularity maximization problem~\prettyref{eq:modularity2} in matrix form as
\begin{equation}
\label{eq:moduarity3}
\begin{aligned}
&\max_{\mtx{Y}} && \iprod{\mtx{Y}}{\mtx{A}- \lambda \vct{d}\vct{d}^\top}
\\
&\text{subject to} && \mtx{Y}\in \textstyle{\bigcup_r}\, \Pnr,
\end{aligned}
\end{equation}
where $ \iprod{\cdot}{\cdot} $ denotes the trace inner product between matrices.
The objective function is linear in matrix variable $\mtx{Y}$, so it suffices to convexify the combinatorial constraint $ \mtx{Y}\in \bigcup_r \Pnr $.

Recall that each matrix $ \mtx{Y} $ in $ \Pnr$ corresponds to a unique partition of $n$ nodes into~$r$ clusters. There is another algebraic representation of such a partition via a \emph{membership matrix} $\mtx{\Psi}  \in \{0,1\}^{ n \times r}$, where $\Psi _{ia}=1$ if and only if node $i$ belongs to the cluster $a$.
These two representations are related by the identity
\begin{equation}
\label{eq:decomposition}
\mtx{Y}= \mtx{\Psi}\mtx{\Psi}^\top,
\end{equation}
which implies that $\mtx{Y}\succeq \mtx{0}$. The membership matrix of a partition is only unique up to permutation of the cluster labels $1, 2, \ldots, r$, so each partition matrix $\mtx{Y}$ corresponds to multiple membership matrices $\mtx{\Psi}$.
We use $ \Mnr $ to denote the set of all possible membership matrices of $ r $-partitions.

Besides being positive semidefinite, a partition matrix $\mtx{Y}$ also satisfies the linear constraints $ 0 \le Y_{ij} \le 1 $ and $Y_{ii}=1$ for all $ i,j \in [n] $. Using these properties of partition matrices, we obtain the following convexification of the modularity optimization problem~\prettyref{eq:moduarity3}:
\begin{equation}
\label{eq:cvx}
\begin{aligned}
\Yhat =\arg\;
&\max_{\mtx{Y}} && \iprod{ \mtx{Y} }{ \mtx{A}- \lambda \vct{d}\vct{d}^\top }
\\
&\text{subject to} && \mtx{Y}\succeq \mtx{0},
\\
&~              && \mtx{0} \leq \mtx{Y} \leq \mtx{J},
\\
&~             &&  Y_{ii} = 1,\quad\text{for~each~} i \in [n].
\end{aligned}
\end{equation}
Here $\mtx{J}$ is the $n \times n$ matrix with all entries equal to $1$. Implementation of the formulation~\prettyref{eq:cvx} requires choosing an appropriate tuning parameter $\lambda$. We will discuss the theoretical range for $\lambda$ for consistent clustering in \prettyref{sec:theory}, and empirical choice of $ \lambda $ in \prettyref{sec:simulation}. As our convexification is based on the generalized version~\prettyref{eq:modularity2} of modularity maximization, it is capable of detecting small clusters, even when the number of clusters $ r $ grows with $ n $, as is shown later.

\subsection{Explicit clustering via weighted $k$-median}
\label{sec:kmedian}

Ideally, the optimal solution $ \Yhat $ to the convex relaxation~\prettyref{eq:cvx} is a valid partition matrix in~$\Pnr $ and recovers the true partition $ \Ystar $ perfectly --- our theoretical results in Section~\ref{sec:theory} characterize when this is the case. In general, the matrix $ \Yhat $ will not lie in $ \Pnr $, but we expect it to be close to $ \Ystar $. To extract an explicit clustering from $ \Yhat $, we introduce a novel and tractable \emph{weighted} $k$-median algorithm.

Recall that by definition, the $ i $-th and $ j $-th rows of the true partition matrix $ \Ystar $ are identical if the corresponding nodes $ i $ and $ j $ belong to the same community, and otherwise orthogonal to each other. If $\Yhat$ is close to $\Ystar$, intuitively one can extract a good partition by clustering the row vectors of $\Yhat$ as points in the Euclidean space $ \reals^n $.
While there exist numerous algorithms (e.g., $ k $-means) for such a task, our analysis identifies a particularly viable choice --- a $ k $-median procedure appropriately weighted by the node degrees --- that is efficient both theoretically and empirically.

Specifically, our weighted $ k $-median procedure consists of two steps. First, we multiply the columns of $\Yhat$ by the corresponding degrees to obtain the matrix $\widehat{\mtx{W}}:= \Yhat\mtx{D}$, where $\mtx{D}:=\diag(\vct{d})=\diag(d_1, \ldots, d_n)$, which is the diagonal matrix formed by the entries of $\vct{d}$. Clustering is performed on the row vectors of $\widehat{\mtx{W}}$ instead of $ \Yhat $.
Note that if we consider the $ i $-th row of $ \Yhat $ as a vector of $ n $ features for node~$ i $, then the rows of $ \widehat{\mtx{W}} $ can be thought of as vectors of \emph{weighted} features.

In the second step, we implement a weighted $k$-median clustering on the row vectors of $ \widehat{\mtx{W}} $. Denoting by $\hat{\vct{w}}_i$ the $ i $-th row of $ \widehat{\mtx{W}} $, we search for a partition $C_1, \ldots, C_r$ of $ [n] $ and $r$ cluster centers $\vct{x}_1, \ldots, \vct{x}_r \in \mathbb{R}^n$ that minimize the sum of the weighted distances in $ \ell_1 $ norm
\[
\sum_{1 \leq a \leq r} \sum_{i \in C_a} d_i \| \hat{\vct{w}}_i - \vct{x}_a\|_1.
\]
Additionally, we require that the center vectors $\vct{x}_1, \ldots, \vct{x}_r$ are chosen from the row vectors of $\widehat{\mtx{W}}$ (these centers are sometimes called \emph{medoids}).

Representing the partition $ \{C_a\}_{a=1}^r $ by a membership matrix $ \mtx{\Psi} \in \Mnr $ and the centers $ \{ \vct{x}_i \}$ as the rows of a matrix $ \mtx{X} \in \mathbb{R}^{r \times n}$, we may write the above two-step procedure compactly as
\begin{equation}
\label{eq:k-median}
\begin{aligned}
&\min_{\mtx{\Psi}, \mtx{X}} &&~\|\mtx{D}(\mtx{\Psi}\mtx{X} - \widehat{\mtx{W}})\|_1
\\
&s.t. && ~\mtx{\Psi} \in \Mnr ,
\\
&~&& ~\mtx{X} \in \mathbb{R}^{r \times n} ,~\text{Rows}(\mtx{X}) \subseteq \text{Rows}(\widehat{\mtx{W}}),
\end{aligned}
\end{equation}
where $\text{Rows}(\mtx{Z})$ denotes the collection of row vectors of a matrix $\mtx{Z}$, and $\|\mtx{Z}\|_1$ denotes the sum of the absolute values of all entries of $\mtx{Z}$.

We emphasize that the formulation~\prettyref{eq:k-median} differs from standard clustering algorithms (such as $ k $-means) in a number of ways. The objective function is the sum of distances rather than that of squared distances (hence $ k $-median), and the distances are in $ \ell_1 $ instead of $ \ell_2 $ norms. Moreover, our formulation has two levels of weighting: each column of $ \Yhat $ is weighted to form $ \widehat{\mtx{W}} $, and the distance of each row $ \vct{w}_i $ to its cluster center is further weighted by $ d_i $.
This doubly-weighted $ \ell_1 $-norm $ k $-median formulation is crucial in obtaining strong and robust statistical bounds, and is significantly different from previous approaches, such as those in~\cite{LR2013,GLM15} (which only use the second weighting, and the weights are \emph{inversely} proportional to $d_i $).
Our double weighting scheme is motivated by the observation that nodes with larger degrees tend to be clustered more accurately --- in particular our analysis of the convex relaxation~\prettyref{eq:moduarity3} naturally leads to a doubly weighted error bound on its solution $ \Yhat $.
On the one hand, for each given $i$, $\hat{Y}_{ij}$ is expected to be closer to $Y^*_{ij}$ if the degree of node $j$ is larger,
so we multiply $\hat{Y}_{ij}$ by $d_j$ for every~$j$ to get the weighted feature vector. On the other hand, the $i$-th
row of $\hat{\mtx{Y} }\mtx{D}$ is closer to the $i$-the row of $\mtx{Y}^*\mtx{D}$ if the degree of node $i$ is larger, hence we
minimize the distances weighted by $d_i$.

With the constraint $\text{Rows}(\mtx{X}) \subseteq \text{Rows}(\widehat{\mtx{W}}) $, the optimization problem~\prettyref{eq:k-median} is precisely the weighted $ \ell_1 $-norm $ k $-median (also known as $ k $-medoid) problem considered in \cite{Charikar02}. Computing the exact optimizer   to~\prettyref{eq:k-median}, denoted by $(\overline{\mtx{\Psi}}, \overline{\mtx{X}})$, is NP-hard.  Nevertheless, \cite{Charikar02} provides a polynomial-time approximation algorithm, which outputs a solution $(\PsiCheck, \XCheck) \in \Mnr \times \mathbb{R}^{r\times n}$ feasible to \prettyref{eq:k-median} and provably satisfying
\begin{align}
\label{eq:approximation}
\|\mtx{D}( \widecheck{ \mtx{\Psi}} \XCheck  - \widehat{\mtx{W}})\|_1 \le \frac{20}{3} \|\mtx{D}( \overline{\mtx{\Psi} } \; \overline{ \mtx{X} } - \widehat{\mtx{W}})\|_1 .
\end{align}

As the solution $ \Yhat $ to the convex relaxation~\prettyref{eq:cvx} and the approximate solution $ \PsiCheck $ to the $ k $-median problem~\prettyref{eq:k-median} can both be computed in polynomial-time,  our algorithm is computationally tractable. In the next section, we  turn to the statistical aspect and show that the clustering induced by $ \Yhat $ and $ \PsiCheck $ is close to the true underlying clusters, under some mild and interpretable conditions of \DCSBM.

\section{Theoretical results}
\label{sec:theory}

In this section, we provide theoretical results characterizing the statistical properties of our algorithm. 
We show that under mild conditions of \DCSBM, the difference between the convex relaxation solution $ \Yhat $ and the true partition matrix $ \Ystar $,
and the difference between the approximate $ k $-median clustering $ \PsiCheck $ and the true clustering $\mtx{\Psi}^*$, are well bounded.
When additional conditions hold, we further show that $ \Yhat $  perfectly recovers the true clusters.
Our results are non-asymptotic in nature, valid for any scaling of $ n$, $r$, $\vct{\theta}$ and $\mtx{B}$ etc.

\subsection{Density gap conditions}
\label{sec:gap}
In the literature of community detection by convex optimization under standard \SBM,  it is often assumed that the minimum within-cluster edge density is greater than the maximum cross-cluster edge density, i.e.,
\begin{equation}
\label{eq:DGC}
\max_{1 \leq a < b \leq r} {B_{ab}}  < \min_{1 \leq a \leq r} {B_{aa} } .
\end{equation}
See for example~\cite{CSX2013,OH2011,AV2011, Cai2014robust, Vershynin14}.
This requirement~\prettyref{eq:DGC} can be directly extended to the \DCSBM setting, leading to the condition
\begin{equation}
\label{eq:SDGC}
 \max_{1 \leq a < b \leq r} \max_{i \in C_a^*, j \in C_b^*} B_{ab} \theta_i \theta_j < \min_{1 \leq a \leq r} \min_{i, j \in C_a^*, i \neq j} B_{aa}\theta_{i}\theta_{j} .
\end{equation}
Under \DCSBM, however, this condition would often be overly restrictive, particularly when the degree parameters $ \{\theta_i\} $ are imbalanced with some of them being very small. In particular, this condition  is highly sensitive to the minimum value $\theta_{\min} :=\min_{1\leq i \leq n} \theta_i,$ which is unnecessary since the community memberships of nodes with larger $ \theta_i $ may still be recoverable.

Here we instead consider a version of the density gap condition that is much milder and more appropriate for \DCSBM. For each cluster index $1 \leq a \leq r$, define the quantities
\begin{equation}
\label{eq:GandH}
G_a := \sum_{i \in C^\ast_a} \theta_i \text{~~~~and~~~~} H_a := \sum_{b=1}^r B_{ab}G_b.
\end{equation}
Simple calculation gives
\[
\E d_i = \theta_i H_a - \theta_i^2 B_{aa} \approx \theta_i H_a.
\]
Therefore, the quantity $H_a$ controls the \emph{average degree} of the nodes in the $ a $-th cluster.
With this notation, we impose the condition that
\begin{equation}
\label{eq:dcgap}
\max_{1 \leq a < b \leq r} \frac{B_{ab}}{H_a H_b} < \min_{1 \leq a \leq r} \frac{B_{aa}}{H_a^2}
\end{equation}
We refer to the condition~\prettyref{eq:dcgap} as the \emph{degree-corrected density gap condition}. This condition can be viewed as the ``average'' version of~\prettyref{eq:SDGC}, as it depends on the aggregate quantity~$ H_a $ associated with each cluster $ a $ rather than the $ \theta_i $'s of individual nodes in the cluster --- in particular, the condition~\prettyref{eq:dcgap} is robust against small $\theta_{\min}$. This condition plays a key role throughout our theoretical analysis, for both approximate and exact cluster recovery under \DCSBM.

To gain intuition on the new degree-corrected density gap condition~\prettyref{eq:dcgap}, consider the following sub-class of \DCSBM with symmetric/balanced clusters.
\begin{definition}
We say that a \DCSBM obeys a $\mathcal{F}(n, r, p, q, g)$-model, if $B_{aa} = p$ for all $a=1, \ldots, r$, $B_{ab}=q$ for all $1\leq a<b \leq r$, and $G_1 = G_2 = \cdots = G_r = g$.
\end{definition}
In a $\mathcal{F}(n, r, p, q, g)$-model, the true clusters are balanced in terms of the connectivity matrix $ \mtx{B} $ and the sum of the degree heterogeneity parameters (rather than the cluster size). Under this model, straightforward calculation gives $H_a = ((r-1)q + p)g$ for all $a = 1 \ldots r$. The degree-corrected density gap condition \prettyref{eq:dcgap} then reduces to $p>q$, i.e., the classical density gap condition \prettyref{eq:DGC}.

\subsection{Theory of approximate clustering}
\label{sec:approx}

We now study when the solutions to our convex relaxation~\prettyref{eq:cvx} and weighted $ k $-median algorithms~\prettyref{eq:k-median} approximately recover the underlying true clusters.
Under \DCSBM, nodes with different $ \theta_i $'s have varying degrees, and therefore contribute differently to the overall graph and in turn to the clustering quality. Such heterogeneity needs to be taken into account in order to get tight bounds on clustering errors. The following version of $ \ell_1 $ norm, corrected by the degree heterogeneity parameters, is the natural notion of an error metric:
\begin{definition}
\label{def:weighted_norm}
For each matrix $\mtx{Z} \in \mathbb{R}^{n \times n}$, its weighted element-wise $ \ell_1 $ norm is defined as
\[
\|\mtx{Z}\|_{1, \vct{\theta}}:= \sum_{1 \leq i, j \leq n} |\theta_i Z_{ij} \theta_j| .
\]
\end{definition}
\noindent Also recall our definitions of $H_a$ and $G_a$ in equation~\prettyref{eq:GandH}. Furthermore, for each $1\leq a \leq r$ and $i \in C^*_a$, define the quantity $f_i := \theta_i H_a$, which corresponds approximately to the expected degree of node $ i $ and satisfies
$\|\vct{f}\|_1 = \sum_{a, b} B_{ab} G_a G_b$.

With the notation above, our first theorem shows that the convex relaxation solution~$ \Yhat $ is close to the true partition matrix~$ \Ystar $ in terms of the weighted $ \ell_1 $ norm.
\begin{theorem}
\label{thm:approx}
Under \DCSBM, assume that the degree-corrected density gap condition \prettyref{eq:dcgap} holds. Moreover, suppose that the tuning parameter $\lambda$ in the convex relaxation~\prettyref{eq:cvx} satisfies
\begin{equation}
\label{eq:RDGC}
\max_{1 \leq a < b \leq r} \frac{B_{ab} + \delta}{H_a H_b} < \lambda < \min_{1 \leq a \leq r} \frac{B_{aa} - \delta}{H_a^2}
\end{equation}
for some number $\delta>0$. Then with probability at least $0.99 - 2(e/2)^{-2n}$, the solution~$ \widehat{ \mtx{Y} } $ to \prettyref{eq:cvx} satisfies the bound
\begin{equation}
\label{eq:pillar1}
\|\Ystar - \widehat{ \mtx{Y} } \|_{1, \vct{\theta}} \leq \frac{C_0}{\delta}\left( 1+ \left(\min_{1 \leq a \leq r} \frac{B_{aa}}{H_a^2}\right)  \|\vct{f}\|_1 \right)\left(  \sqrt{n\|\vct{f}\|_1 }  + n \right),
\end{equation}
where $ C_0 >0 $  is an absolute constant.
\end{theorem}
\noindent We prove this claim in~\prettyref{sec:proof_approx}. The bound~\prettyref{eq:pillar1} holds with probability close to one. Notably, it is insensitive to $\theta_{\min}$ as should be expected, because community memberships of  nodes with relatively large $\theta_i$
are still recoverable. In contrast, the error bounds of several existing methods, such as that of SCORE method in \cite[eq.\ (2.15), (2.16)] {Jin2012}, depend on $\theta_{\min}$ crucially.

Under the $\mathcal{F}(n, r, p, q, g)$-model, recall that  $ H_a \equiv (p+(r-1)q)g$ and density gap condition~\prettyref{eq:dcgap} becomes $p>q$. Moreover, the constraint \prettyref{eq:RDGC} for $\delta$ and $\lambda$ becomes
\begin{equation}
\label{eq:RDGC2}
p-q > 2\delta \quad \text{and} \quad \frac{q+\delta}{(p+(r-1)q)^2g^2} < \lambda  < \frac{p - \delta}{(p+(r-1)q)^2g^2}.
\end{equation}
Note that the first inequality above is the same as the standard density gap condition imposed in, for example, \cite{CSX2013, ChenXu14, Cai2014robust}. Furthermore, the vector $ \vct{f} $ satisfies $\|\vct{f}\|_1 = r(p+(r-1)q)g^2 \le r^2 p g^2$. Substituting these expressions into the bound~\prettyref{eq:pillar1}, we obtain the following corollary for the symmetric \DCSBM setting.
\begin{corollary}
\label{cor:Fmodel}
Under the  $\mathcal{F}(n, r, p, q, g)$-model of \DCSBM, if the condition~\prettyref{eq:RDGC2} holds for the density gap and tuning parameter, then with probability at least $0.99 - 2(e/2)^{-2n}$, the solution~$ \widehat{ \mtx{Y} } $ to the convex relaxation~\prettyref{eq:cvx} satisfies the bound
\begin{equation}
\label{eq:pillar2}
\|\Ystar - \widehat{ \mtx{Y} } \|_{1, \vct{\theta}}  \lesssim \frac{1}{\delta} \left( 1+ \frac{r p}{ ( p+(r-1)q ) } \right)  (n + rg\sqrt{np})
\lesssim \frac{1}{\delta} r (n + rg\sqrt{np}).
\end{equation}
\end{corollary}

Note that if $ \frac{p}{q}  = c$ for an absolute constant $ c$, then the bound~\prettyref{eq:pillar2} takes the simpler form $ \|\Ystar - \widehat{ \mtx{Y} } \|_{1, \vct{\theta}}  \lesssim \frac{n + rg\sqrt{np}}{\delta} $. If $\theta_i=1$ for all nodes $i$, the $\mathcal{F}(n, r, p, q, g)$-model reduces to the standard \SBM with equal community size. If we assume $r=O(1)$ additionally, and note that $g=n/r$ and let $\delta = \frac{p-q}{4}$, then the error bound \prettyref{eq:pillar2} becomes
\begin{align*}
 \|\Ystar - \widehat{ \mtx{Y} } \|_{1}  \lesssim \frac{n( 1 + \sqrt{np} )  }{p-q}.
\end{align*}
This bound matches the error bounds proved in \cite[Theorem 1.3]{Vershynin14}. \\

The output $ \Yhat $ of the convex relaxation need not be a partition matrix corresponding to a clustering; a consequence is that the theoretical results in~\cite{Vershynin14} do not provide an explicit guarantee on clustering errors (except for the special case of $ r=2 $). We give such a bound below, based on the explicit clustering extracted from $ \Yhat $ using the weighted $ \ell_1 $-norm $ k $-median algorithm~\prettyref{eq:k-median}. Recall that $\PsiCheck$ is the membership matrix in the approximate $ k $-median solution given in~\prettyref{eq:approximation}, and let $\mtx{\Psi}^\ast$ be the membership matrix corresponding to the true clusters. A membership matrix is unique only up to permutation of its columns (i.e., relabeling the clusters), so counting the misclassified nodes in $ \PsiCheck $ requires an appropriate minimization over such permutations. The following definition is useful to this end. For a matrix $\mtx{M}$,  let $\mtx{M}_{i \bullet}$ denote its $i$-th row vector. 

\begin{definition}
\label{def:calE}
Let $\mathcal{S}_r$ denote the set of all $r \times r$ permutation matrices. The set of misclassified nodes with respect to a permutation matrix $\mtx{\Pi} \in \mathcal{S}_r $ is defined as
\begin{align*}
\mathcal{E}(\mtx{\Pi}) := \big\{ i \in [n] : \big( \PsiCheck \mtx{\Pi} \big)_{i \bullet} \neq \mtx{\Psi}^\ast_{i \bullet} \big\}.
\end{align*}
\end{definition}

\noindent With this definition, we have the following theorem that quantifies the misclassification rate of approximate $ k $-median solution $ \PsiCheck $.
\begin{theorem}
\label{thm:kmedian}
Under the $\mathcal{F}(n, r, p, q, g)$-model, assume that the parameters $\delta$ and $\lambda$ satisfy \prettyref{eq:RDGC2}. Then with probability at least $0.99 - 2(e/2)^{-2n}$, the approximate $ k $-median solution $ \PsiCheck $ satisfies the bound
\begin{align}
\label{eq:Sbound}
\min_{\mtx{\Pi} \in \mathcal{S}_r}  \Big\{  \textstyle{\sum_{i \in \mathcal{E}(\mtx{\Pi})} \theta_{i} } \Big\} 
&\leq C_0\frac{r}{\delta} \left( \frac{n}{g} + r\sqrt{np} \right)
\end{align}
for some absolute constant $C_0$.
\end{theorem}
\noindent We prove this claim in \prettyref{sec:proof_kmedian}. Extension to the general \DCSBM setting is straightforward.

If we let $\mtx{\Pi}_{ \vct{\theta} }$ be a minimizer of the LHS of~\prettyref{eq:Sbound} and $\mathcal{E}_{\vct{\theta}} := \mathcal{E}(\mtx{\Pi}_{ \vct{\theta} })$, then the quantity $ \sum_{i \in \mathcal{E}_{\vct{\theta}} } \theta_i $ is the number of misclassified nodes \emph{weighted} by their degree heterogeneity parameters $ \{ \theta_i \}$. \prettyref{thm:kmedian} controls this weighted quantity, which is natural as nodes with smaller $ \theta_i $ are harder to cluster and thus less controlled in~\prettyref{eq:Sbound}. 
Notably, the bound given in \prettyref{eq:Sbound}  is applicable even in the \emph{sparse graph regime} with bounded average degrees, i.e., $p, q = O(1/n)$. For example, suppose that $p =a/n$ and $q=b/n$ for two fixed constants $a>b$,  $r=O(1)$ and $g \asymp n$; if $(a-b)/\sqrt{a}$ is sufficiently large, then, with the choice $ \delta \asymp (a-b)/n$, the right hand side of~\prettyref{eq:Sbound} can be an arbitrarily small constant times $ n $.
In comparison, conventional spectral methods are known to be inconsistent in this sparse regime \citep{KMMNSZZ2013}. While this difficulty is alleviated under \SBM by the use of regularization or non-backtracking matrices (\eg, \cite{LeVershynin15,BordenaveLelargeMassoulie:2015dq}),
rigorous justification and numerical validation under \DCSBM have not been well explored.
\\

It is sometimes desirable to have a direct (unweighted) bound on the number misclassified nodes. Suppose that $\mtx{\Pi}_0 \in \mathcal{S}_r $ is a permutation matrix that minimizes $|\mathcal{E}(\mtx{\Pi})|$, and let $\mathcal{E}_0:=\mathcal{E}(\mtx{\Pi}_0)$. A bound on the unweighted misclassification error $ | \mathcal{E}_0 | $ can be easily derived from the general weighted bound~\prettyref{eq:Sbound}. For example, combining~\prettyref{eq:Sbound} with the AM-HM inequality $\sum_{i\in \mathcal{E}_{\vct{\theta}}} \theta_{i}\ge\frac{|\mathcal{E}_{\vct{\theta}}|^{2}}{\sum_{i\in \mathcal{E}_{\vct{\theta}}} 1/\theta_{i}} $,
we obtain that
\begin{equation}
|\mathcal{E}_0| \leq |\mathcal{E}_{\vct{\theta}} | 
\lesssim \sqrt{\frac{1}{\delta g} r (n + rg\sqrt{np}) \sum_{i=1}^{n} \frac{1}{\theta_{i}}}.
\label{eq:Sbound_AM_HM}
\end{equation}
Another bound on $ |\mathcal{E}_0| $, which is applicable even when some $ \theta_i $'s are zero, can be derived as follows: we pick any number $\tau>0$ and use the inequality~\prettyref{eq:Sbound} to get
\begin{equation}
|\mathcal{E}_0| \leq |\mathcal{E}_\theta| \le \frac{1}{\delta g \tau} r (n + rg\sqrt{np}) + \big|\{i:\theta_{i}<\tau\}\big| , \quad \forall \tau >0.
\label{eq:Sbound_tau}
\end{equation}
 This simple bound is already quite useful, for example in standard \SBM with $ \theta_i \equiv 1 $, $ p\ge \frac{1}{n} $ and~$ r $ equal-sized clusters.
In this case, setting $ \tau=0.9 $ in~\prettyref{eq:Sbound_tau} yields that the number of misclassified nodes satisfies  $|\mathcal{E}_0| \lesssim \frac{ r^2 \sqrt{np}  }{ p-q }.$
When $ r=2 $, this bound is consistent with those in~\cite{Vershynin14}, but our result is more general as it applies to more clusters $ r \ge 3 $.

\subsection{Theory of perfect clustering}
\label{sec:exact}

In this section, we show that under an additional condition on the minimum degree heterogeneity parameter $\theta_{\min}= \min_{1 \leq i \leq n} \theta_j$, the solution $ \Yhat $ to the convex relaxation perfectly recovers the true partition matrix $ \Ystar $. In this case the true clusters can be extracted easily from $ \Yhat $ without using the $ k $-median procedure.

For the purpose of studying perfect clustering, we consider a setting of \DCSBM with $B_{aa} = p$ for all $a=1, \ldots, r$, and $B_{ab}=q$ for all $1\leq a<b \leq r$. Under this setup, the degree-corrected density gap condition \prettyref{eq:RDGC} becomes
\begin{equation}
\label{eq:sep_perf}
\max_{1 \leq a < b \leq r} \frac{q + \delta}{H_a H_b} < \lambda < \min_{1 \leq a \leq r} \frac{p - \delta}{H_a^2}.
\end{equation}
Recalling the definition of $ G_a $ in~\prettyref{eq:GandH}, we further define 
$G_{\min}:= \min_{1 \leq a \leq r} G_a.$
The following theorem characterizes when perfect clustering is guaranteed.
\begin{theorem}
\label{thm:exact}
Suppose that the degree-corrected density gap condition \prettyref{eq:sep_perf} is satisfied for some number $\delta >0$ and tuning parameter $\lambda$, and that
\begin{align}
 \label{eq:assump}
 \delta > C_0 \left( \frac{\sqrt{q n}}{G_{\min} }  + \sqrt{\frac{p \log n}{G_{\min}\theta_{\min}}} \right)
\end{align}
for some sufficiently large absolute constant $C_0$. Then with  probability at least $1-10n^{-1}$, the solution~$ \Yhat  $ to the convex relaxation~\prettyref{eq:cvx} is unique and equals $\Ystar$.
\end{theorem}

The condition~\prettyref{eq:assump} depends on the minimum values $G_{\min}$ and $\theta_{\min}$. Such dependence is necessary for perfect clustering, as clusters and nodes with overly small $ G_a $ and $ \theta_i $ will have too few edges and are not recoverable. In comparison, the approximate recovery results in~\prettyref{thm:approx} are not sensitive to either $\theta_{\min}$ or $G_{\min}$, as should be expected.
Valid for the more general \DCSBM, \prettyref{thm:exact} significantly generalizes the existing theory for standard \SBM on perfect clustering by SDP in the literature (see, \eg, \cite{CSX2013, ChenXu14, Cai2014robust}).  Taking $ n\to \infty $, \prettyref{thm:exact} guarantees that the probability of perfect clustering converges to one, thereby implying the convex relaxation approach is \emph{strongly consistent} in the sense of \cite{ZLZ2012}.

In the special case of standard \SBM with $\theta_i=1, \forall i \in [n]$, the density gap lower bound~\prettyref{eq:assump} simplifies to
\begin{align*}
\delta \gtrsim \frac{ \sqrt{qn} } { \ell_{\min} }  + \sqrt{ \frac{p \log n }{ \ell_{\min} } } , 
\label{eq:suff_sdp_exact_sbm}
\end{align*}
where $\ell_{\min} := \min_{1 \le a \le r} \ell_a$ is the minimum community size and $\ell_a := |C^*_a|$ is the size of community $a$.
This density gap lower bound is consistent with best existing results given in \cite{CSX2013, ChenXu14, Cai2014robust} --- as we discussed earlier, our density condition in~\prettyref{eq:RDGC2} under the $\calF(n,r,p,q,g)$ model (which encompasses \SBM with equal-sized clusters) is the same as in these previous papers, with the minor difference that in these papers the term $ \frac{\lambda}{d_i d_j} $ in the convex relaxation~\prettyref{eq:cvx} is replaced by a tuning parameter~$ \lambda' $ assumed to satisfy the condition $q+\delta < \lambda' < p - \delta$.

\section{Numerical results}
\label{sec:simulation}

In this section, we provide numerical results on both synthetic and real datasets, which corroborate our theoretical findings.
Our convexified modularity maximization approach is found to empirically outperform state-of-the-art methods in several settings. 

The convexified modularity maximization problem~\prettyref{eq:cvx} is a semidefinite program (SDP), and can be solved efficiently by a range of general and specialized algorithms. Here we use the alternating direction method of multipliers (ADMM) suggested in \cite{Cai2014robust}.
To specify the ADMM solver, we need some notations as follows. For any two $n \times n$ matrices $\mtx{X}$ and $\mtx{Y}$, let $\max\{ \mtx{X}, \mtx{Y} \}$ be the matrix whose $(i,j)$-th entry is given by
$\max\{X_{ij}, Y_{ij}\}$; the matrix $ \min\{ \mtx{X}, \mtx{Y} \}$ is similarly defined.
For a symmetric matrix $\mtx{X}$ with an eigenvalue decomposition $\mtx{X}= \mtx{U} \mtx{\Sigma}  \mtx{U}^\top$,
let $( \mtx{X} )_{+} :=  \mtx{U} \max \{ \mtx{\Sigma}, \mtx{0} \} \mtx{U}^\top$,
and let $(\mtx{X} )_{\mtx{I}}$ be the matrix obtained by setting all the diagonal entries of $\mtx{X}$ to $1$. Recall that $\mtx{J}$ denotes the $n \times n$ all-one matrix. The ADMM algorithm for solving~\prettyref{eq:cvx} with the dual update step size equal to $1$, is given as \prettyref{alg:ADMMSDP}.

\begin{algorithm}
\caption{ADMM algorithm for solving the SDP \prettyref{eq:cvx} } \label{alg:ADMMSDP}
\begin{algorithmic}[1]
\STATE Input: $\mtx{A}$ and $\lambda= \Iprod{\mtx{A}}{\mtx{J}}^{-1}$.
\STATE Initialization:   $ \mtx{Z}^{(0)} = \mtx{\Lambda}^{(0)} = 0$, $ k=0 $ and $\texttt{MaxIter} =100$.
\STATE  while $k < \texttt{MaxIter}$
\begin{enumerate}
\item $
\mtx{Y}^{(k+1)}  = \left(\mtx{Z}^{(k)} - \mtx{\Lambda}^{(k)} + \mtx{A} - \mtx{\lambda} \vec{d} \vec{d}^\top \right)_{+}
$
\item $
\mtx{Z}^{(k+1)} = \left(  \min \left\{ \max \left\{ \mtx{Y}^{(k) } + \mtx{\Lambda}^{(k)}, \mtx{0} \right\} , \mtx{J} \right\} \right)_{\mtx{I}}
$
\item $
\mtx{\Lambda}^{(k+1)}  = \mtx{\Lambda}^{(k)} + \mtx{Y}^{(k+1)}- \mtx{Z}^{(k+1)}
$
\item $ k = k+1 $
\end{enumerate}
end while
\STATE Output the final solution $\mtx{Y}^{(k)}$.
\end{algorithmic}
\end{algorithm}

Our choice of the tuning parameter $\lambda= \Iprod{\mtx{A}}{\mtx{J}}^{-1}$ is motivated by the following simple observation.
By standard concentration inequalities,  the number $\Iprod{\mtx{A}}{\mtx{J}}$ is close to its expectation $\sum_{i} \expect{d_i} \approx \| \vct{f} \|_1$.
Under the $\calF(n,r,p,q,g)$-model, we have $\| \vct{f} \|_1 = r (p + (r-1) q ) g^2$  and $H_a=( (r-1)q +p ) g$ for all $a \in [r]$.
In this case and with the above choice of $ \lambda $, the density gap assumption~\prettyref{eq:sep_perf} simplifies to $ \frac{q+\delta}{(r-1)q+p} < \frac{1}{r} <\frac{p-\delta}{ (r-1) q +p }$, which holds with $\delta=(p-q)/r$.

After obtaining the solution $\Yhat$ to the convex relaxation, we extract an explicit clustering using the weighted $k$-median procedure described in~\prettyref{eq:k-median} with $ k=r $, where the number of major clusters $r$ is assumed known.  Our complete community detection algorithm, Convexified Modularity Maximization (CMM), is summarized in~\prettyref{alg:CD}. In our experiments, the weighted $ k $-median problem is solved by an iterative greedy procedure that optimizes alternatively over the variables $ \mtx{\Psi} $ and $ \mtx{X} $ in~\prettyref{eq:k-median}, with $ 100 $ random initializations.

\begin{algorithm}
\caption{Convexified Modularity Maximization (CMM)} \label{alg:CD}
\begin{algorithmic}[1]
\STATE Input: $\mtx{A}$, $\lambda= \Iprod{\mtx{A}}{\mtx{J}}^{-1}$, and $r \ge 2$.
\STATE Solve the convex relaxation~\prettyref{eq:cvx} for $\widehat{\mtx{Y}}$ using \prettyref{alg:ADMMSDP}.
\STATE Solve the weighted $ k $-median problem~\prettyref{eq:k-median} with $\widehat{\mtx{W}}=\widehat{\mtx{Y}}\mtx{D}$ and $k = r$, and output the resulting $ r $-partition of $ [n] $.
\end{algorithmic}
\end{algorithm}

\subsection{Synthetic data experiments}
\label{sec:synthetic}

We provide experiment results on synthetic data generated from \DCSBM. For each node $ i \in [n]$, the degree heterogeneity parameter $ \theta_i $ is sampled independently from a Pareto$ (\alpha, \beta) $ distribution with the density function
$f(x | \alpha, \beta) =  \frac{\alpha \beta^\alpha}{x^{\alpha + 1}} \indc{x \ge \beta}$, 
where $\alpha$ and $\beta$ are called the \emph{shape} and \emph{scale} parameters, respectively. We consider different values of the shape parameter, and choose the scale parameter accordingly so that the expectation of each $ \theta_i $ is fixed at $ 1 $. Note that the variability of the $\theta_i $'s decreases with the shape parameter $ \alpha $.  Given the degree heterogeneity parameters $ \{\theta_i\} $ and two numbers $0 \le q<p \le 1 $, a graph is generated from \DCSBM, with the edge probability between nodes $i \in C^*_a $ and $j \in C^*_b$ being $\min(1, \theta_i \theta_j B_{ab})$ and $ B_{aa}  = p,  B_{ab} = q, \forall 1\le a \neq b\le r $. 

We applied our CMM approach in \prettyref{alg:CD} to the resulting graph, and recorded the misclassification rate $ | \calE(\mtx{\Pi}_0) | /n $ (cf.\ the discussion after \prettyref{thm:kmedian}).
For comparison, we also applied the SCORE algorithm in~\cite{Jin2012} and the OCCAM algorithm in~\cite{ZLZ14}, which are reported to have state-of-the-art empirical performance on \DCSBM in the existing literature. The SCORE algorithm performs $ k $-means on the  top-$ 2 $ to top-$ r $  eigenvectors of the adjacency matrix normalized element-wise by the top-$ 1 $ eigenvector. OCCAM is a type of regularized spectral $ k $-median algorithm; it can produce non-overlapping clusters and its regularization parameter is given explicitly in~\cite{ZLZ14}. For all $ k $-means/medians procedures used in the experiments, we took  $ k=r $ and used $ 100 $ random initializations.

\begin{figure}[htbp!]
	\centering
	\begin{tabular}{cc}
		\includegraphics[width=0.45\columnwidth]{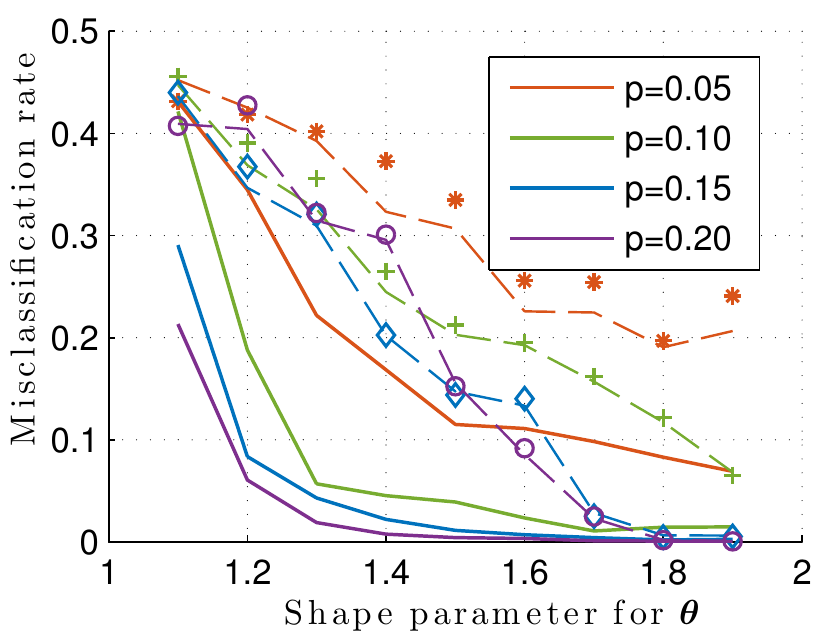}&  
		\includegraphics[width=0.45\columnwidth]{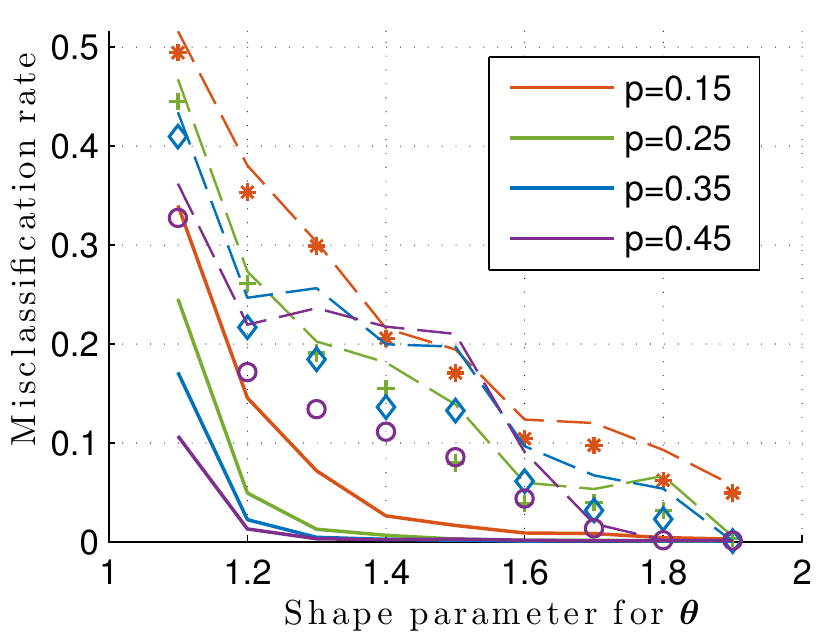}\\
		(a) & (b) \\
		\includegraphics[width=0.45\columnwidth]{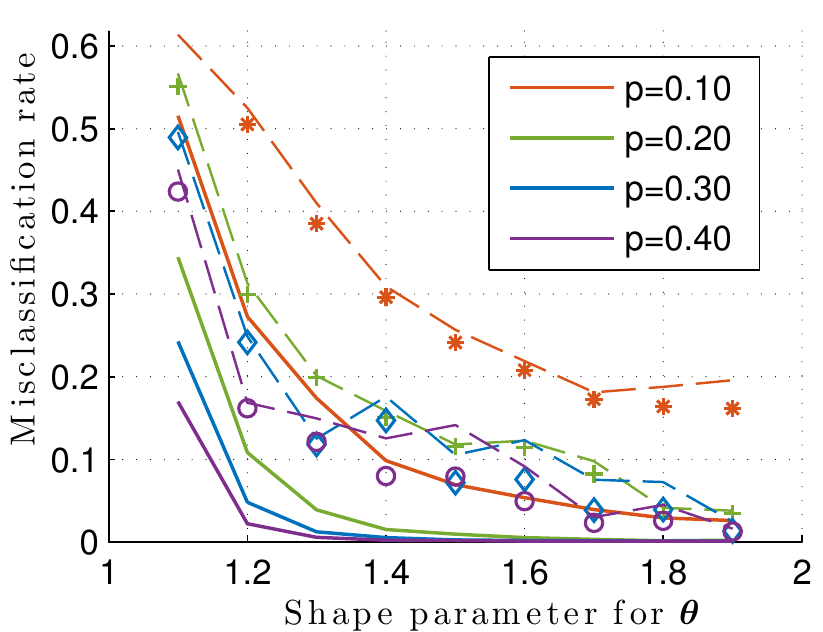} & 
		\includegraphics[width=0.45\columnwidth]{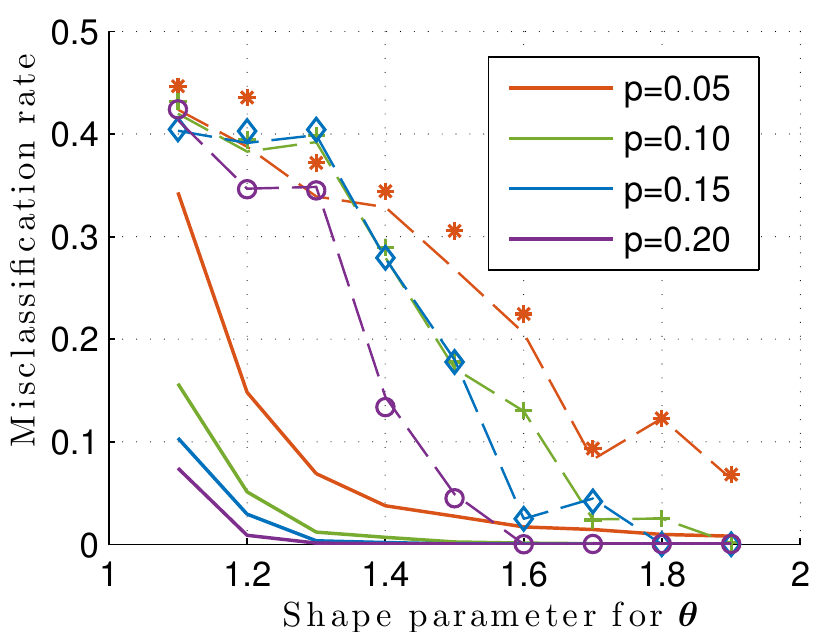}\\ 
		(c) & (d)\\
	\end{tabular}
	\caption{Clustering performance on synthetic datasets versus the variability of $ \vct{\theta} $. Solid lines: our CMM algorithm. Dash lines: the SCORE algorithm. Individual marker: Regularized spectral algorithm.  Panel (a): $ 400 $ nodes, $ 2 $ clusters of size $ 200 $. 
	Panel (b): $ 600 $ nodes, $ 3 $ clusters of size $ 200 $.
	Panel (c): $ 800 $ nodes, $ 4 $ clusters of size $ 200 $. 
	Panel (d): $ 900 $ nodes, $ 2 $ clusters of size $ 450 $.
	Each point represents the average of $ 20 $ independent runs. In all experiments we set $ q = 0.3p $.  }
	\label{fig:synthetic}
\end{figure}

\prettyref{fig:synthetic} shows the misclassification rates of CMM (solid lines), SCORE (dash lines) and OCCAM (individual markers) for various settings of $ n $, $ p $, $ q $, cluster size and the shape parameter for $ \vct{\theta} $. 
We see that the misclassification rate of CMM grows as the degree parameters $ \{ \theta_i \} $ becomes more heterogeneous (smaller values of the shape parameter),  and as the graph becomes sparser, which is consistent with the prediction of \prettyref{thm:kmedian}. Moreover, our approach has consistently lower misclassification rates than SCORE and OCCAM, with SCORE and OCCAM exhibiting similar performance.

\subsection{Political blog network dataset}
\label{sec:poliblog}

We next test the empirical performance of CMM (\prettyref{alg:CD}), SCORE and OCCAM
on the US political blog network dataset from~\cite{AG2005}.
This dataset consists of 19090 hyperlinks (directed edges) between 1490 political blogs collected in the year 2005.
The political leaning (liberal versus conservative) of each blog
has been labeled manually based on blog directories, incoming and outgoing links and posts around the time of the 2004 presidential election. We treat these labels as the true memberships of $ r=2 $ communities. We ignore the edge direction, and focus on the largest connected component with $ n = 1222$ nodes and $16,714$ edges, represented by the adjacency matrix $ \mtx{A} $.
This graph has high degree variation: the maximum degree is $351$ while the mean degree is around $27$.
CMM, SCORE and OCCAM misclassify $62$, $ 58 $ and $ 65 $ nodes, respectively, out of $1222$ nodes on this dataset. The SCORE method has the best known error rate  on the political blogs dataset in the literature \citep{Jin2012}, and we see that our CMM approach is comparable to the state of the art. Panel~(a) in \prettyref{fig:sol} shows the adjacency matrix $ \mtx{A} $ with rows and columns sorted according to the true community labels.
The output of ADMM \prettyref{alg:ADMMSDP} for solving the convex relaxation~\prettyref{eq:cvx} is shown in~\prettyref{fig:sol}~(b). The partition matrix corresponding to the output of the weighted $ k $-median step in~\prettyref{alg:CD} is shown in~\prettyref{fig:sol}~(c). 

\begin{figure}[ht]
\centering
\begin{tabular}{ccc}
\includegraphics[width=.3\columnwidth]{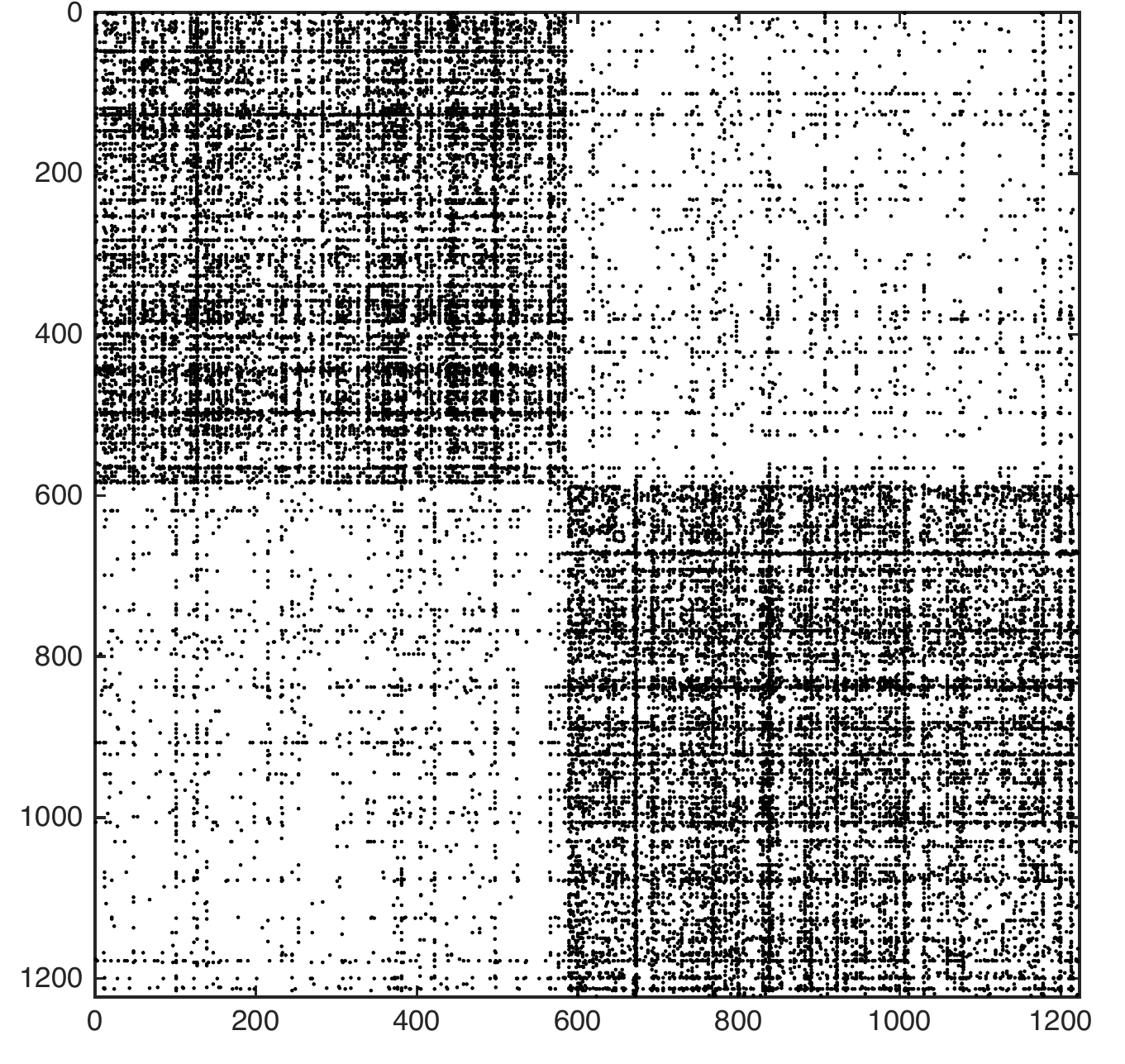} & \includegraphics[width=.31\columnwidth]{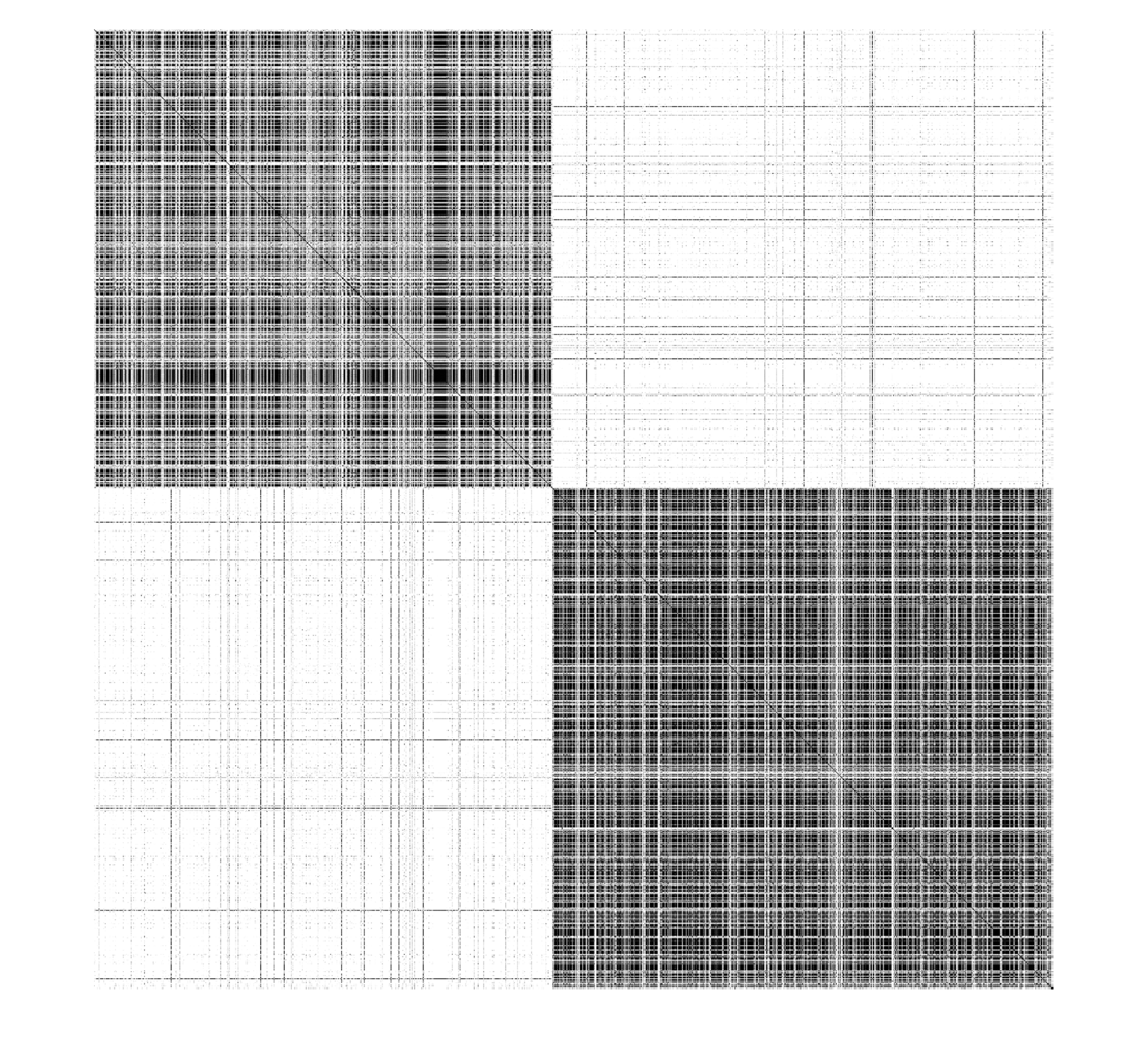}&  \includegraphics[width=.31\columnwidth]{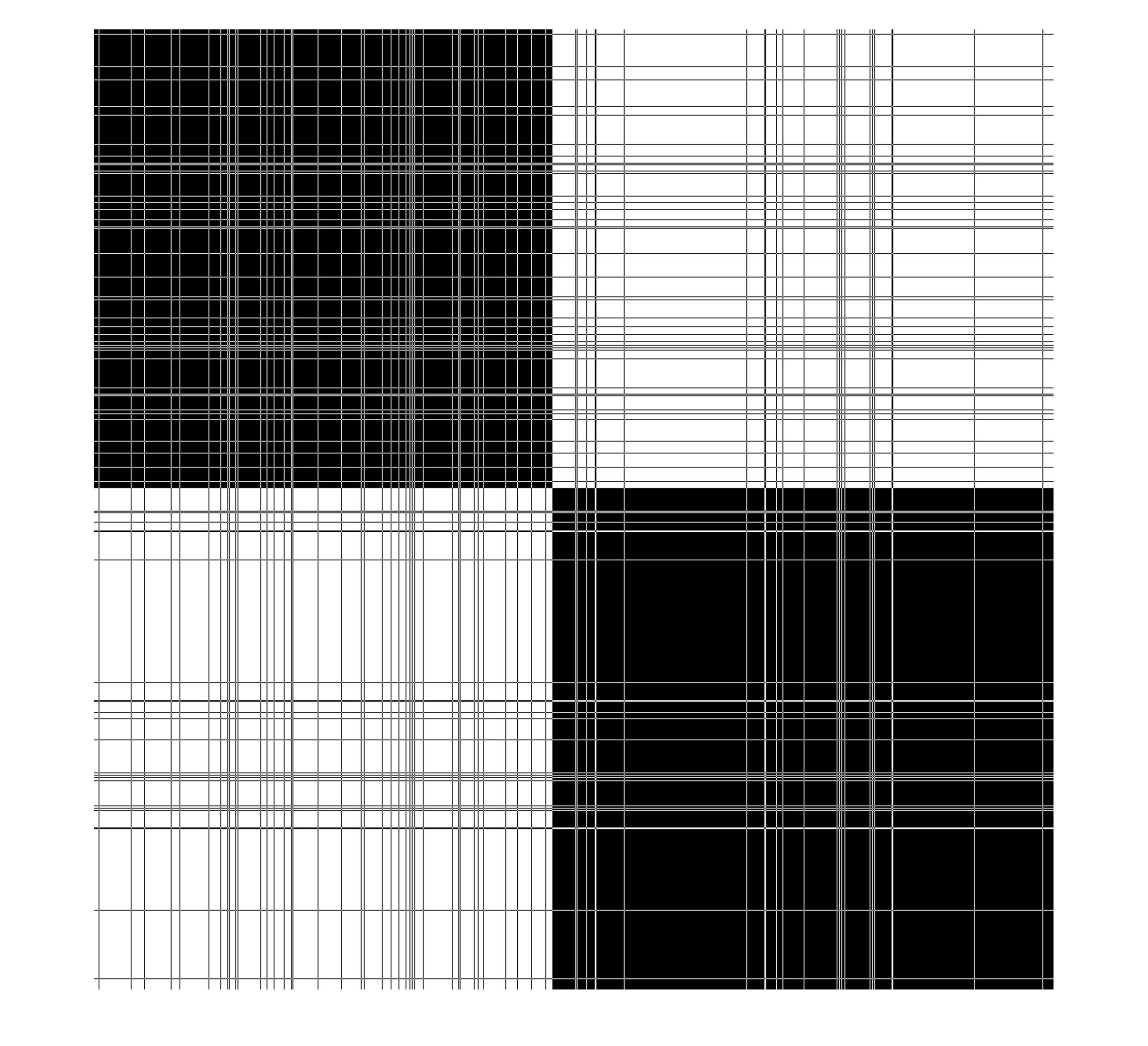}\\
(a)	& (b)  & (c)\\
\end{tabular}
\caption{Panel (a): The adjacency matrix of the largest connected component in the political blog network with $ 1222 $ nodes. The rows and columns are sorted according to the true community labels. The first $ 586 $ rows/columns correspond to the liberal community, and the next $636$ is the conservative community. 
Panel (b): The output matrix $ \Yhat $ of the convex relaxation~\prettyref{eq:cvx} solved by ADMM (\prettyref{alg:ADMMSDP}), with the entries  truncated to the interval $ [0,1] $. Panel~(c): The partition matrix corresponding to the output of CMM (\prettyref{alg:CD}), obtained from the weighted $ k$-median algorithm. Matrix entry values are shown in gray scale with black corresponding to $1$ and white to $ 0 $.}
\label{fig:sol}
\end{figure}

\subsection{Facebook dataset}
\label{sec:facebook}

In this section, we consider the Facebook network dataset from \cite{facebookdataset11,Traud2012}, and compare the empirical performance
of our CMM approach with the SCORE and OCCAM methods.
The Facebook network dataset consists of 100 US universities and all the ``friendship'' links between the users within each university, recorded on one particular day in September 2005. The dataset also contains several node attributes such as the gender, dorm, graduation year and academic major of each user. Here we report results on the friendship networks
of two universities: Simmons College and Caltech.

\subsubsection{Simmons College network}

The Simmons College network contains 1518 nodes and 32988 undirected edges.
The subgraph induced by nodes with graduation year between 2006 and 2009 has a largest connected component with 1137 nodes and 24257 undirected edges, which we shall focus on.
It is observed in \cite{facebookdataset11,Traud2012} that the community structure of the Simmons College network exhibits
a strong correlation with the graduation year --- students in the same year are more likely to be friends.
Panel (a) of \prettyref{fig:simmons} shows this largest component with nodes colored according to their graduation year.

We applied the CMM (\prettyref{alg:CD}), SCORE and OCCAM methods to partition the largest component into $ r=4 $ clusters. In Panels (b)--(d) of \prettyref{fig:simmons} the clustering results of these three methods are shown as the node colors.
In \prettyref{fig:simmons_confmat} we also provide the confusion matrices of the clustering results against the graduation years; the $ (i,j) $-th entry of a confusion matrix represents the number of nodes that are from graduation year~$ i + 2005 $ but assigned to cluster~$ j $ by the algorithm.
We see that our CMM approach produced a partition more correlated with the actual graduation years.  
In fact, if we treat the graduation years as the ground truth cluster labels, then CMM misclassified $ 12.04\% $ of the nodes, whereas SCORE and OCCAM have higher misclassification rates of $ 23.57\% $ and $22.43\%$, respectively. A closer investigation of \prettyref{fig:simmons}  and \prettyref{fig:simmons_confmat} shows that CMM was better in distinguishing between the nodes of year 2006 and 2007.

 \begin{figure}[htbp!]
\centering
\begin{tabular}{cc}
\includegraphics[width=.45\columnwidth]{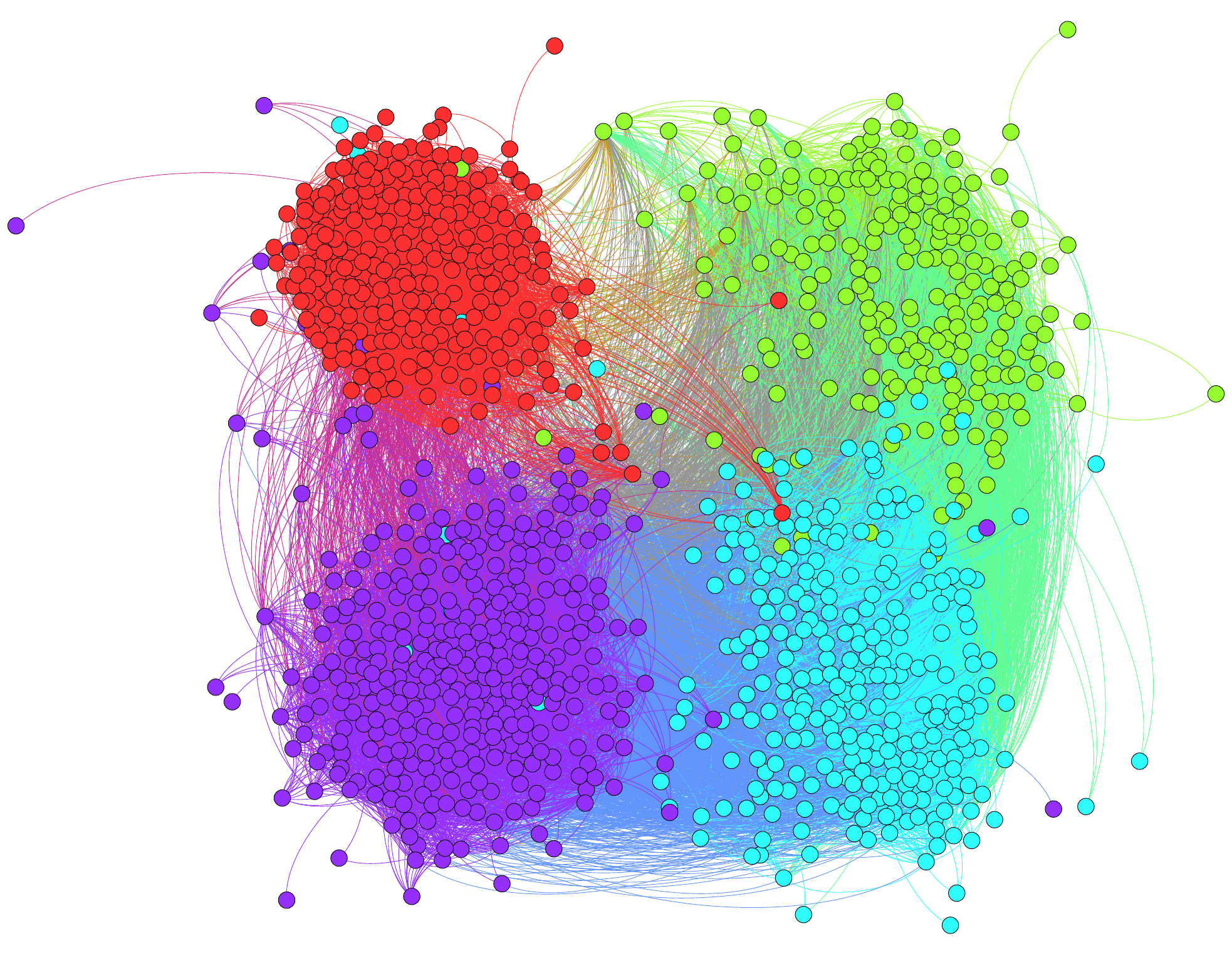}	&  \includegraphics[width=.45\columnwidth]{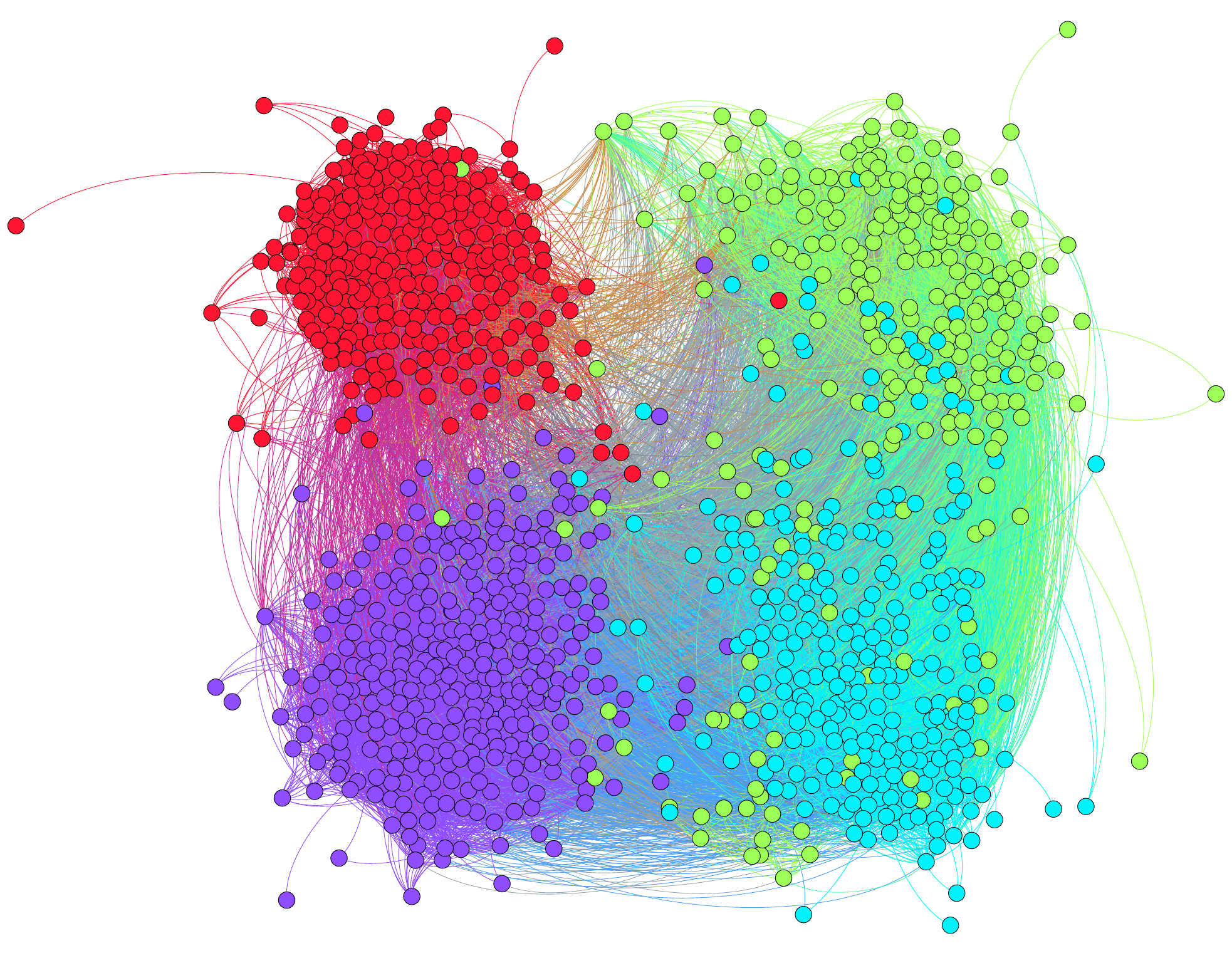}\\
(a)	& (b)  \\
\includegraphics[width=.45\columnwidth]{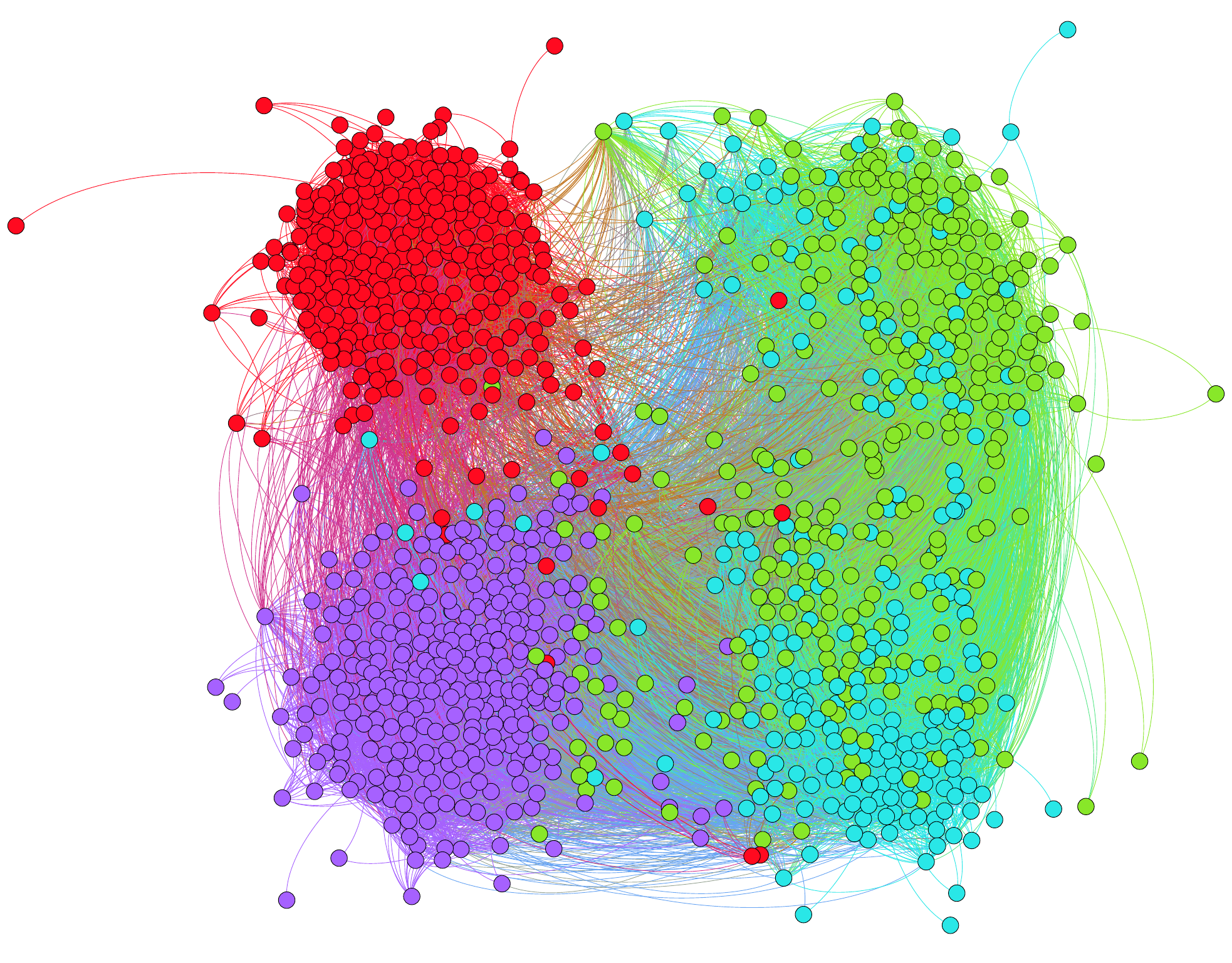}& 
\includegraphics[width=.45\columnwidth]{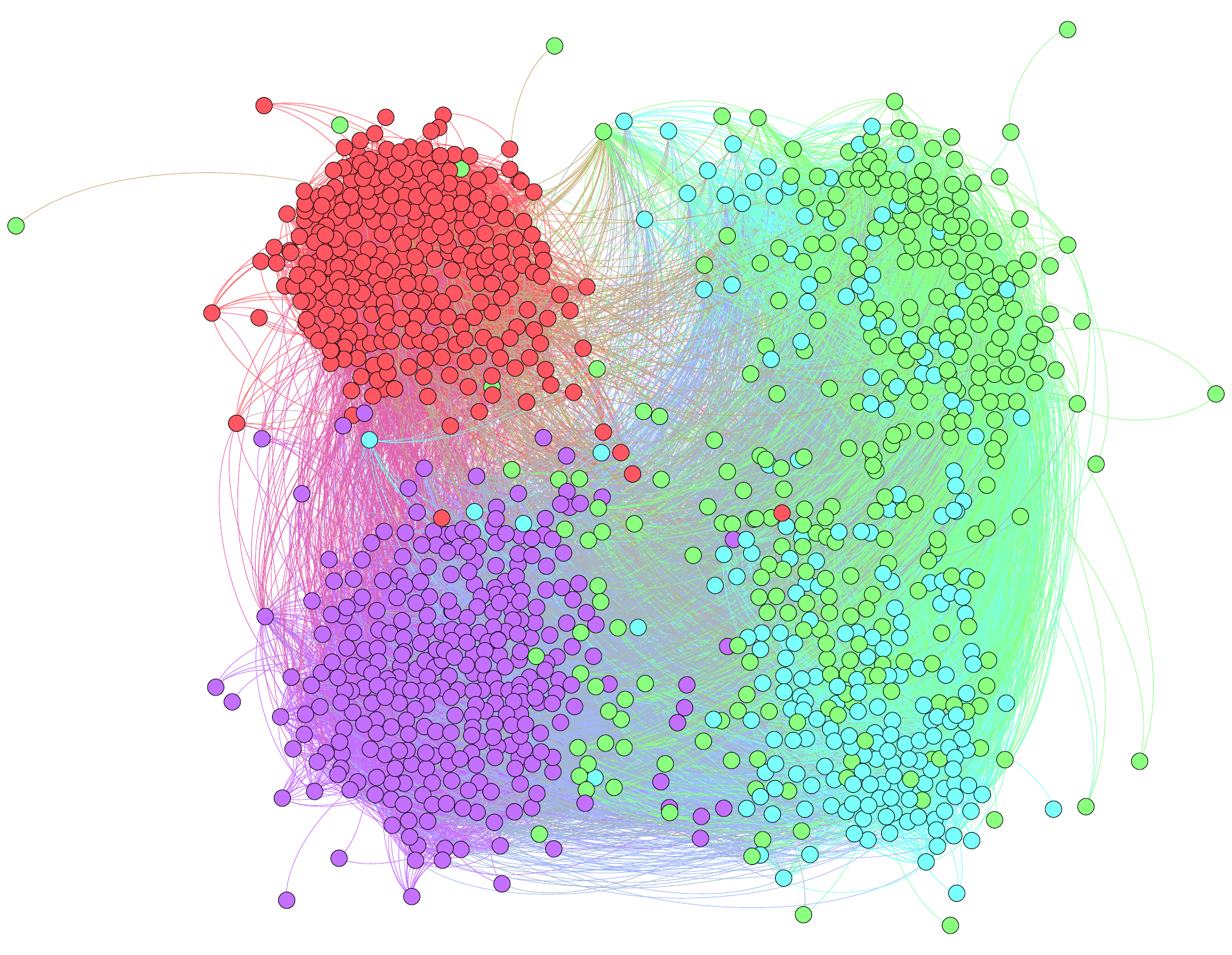}\\
(c) & (d)
\end{tabular}
\caption{ The largest component of the Simmons College network. Panel (a): Each node is colored according to its graduation year, with 2006 in green, 2007 in light blue,  2008 in purple and 2009 in red.
Panels (b)--(d): Each node is colored according to the clustering result of (b) CMM, (c) SCORE and (d) OCCAM. (These plots are generated using the Gephi package \citep{Gelphi15} with the ForceAtlas2 layout algorithm \citep{ForceAtlas14}.)}
\label{fig:simmons}
\end{figure}

\begin{figure}[htbp!]
\centering
\begin{tabular}{ccc}
\includegraphics[width=.31\columnwidth]{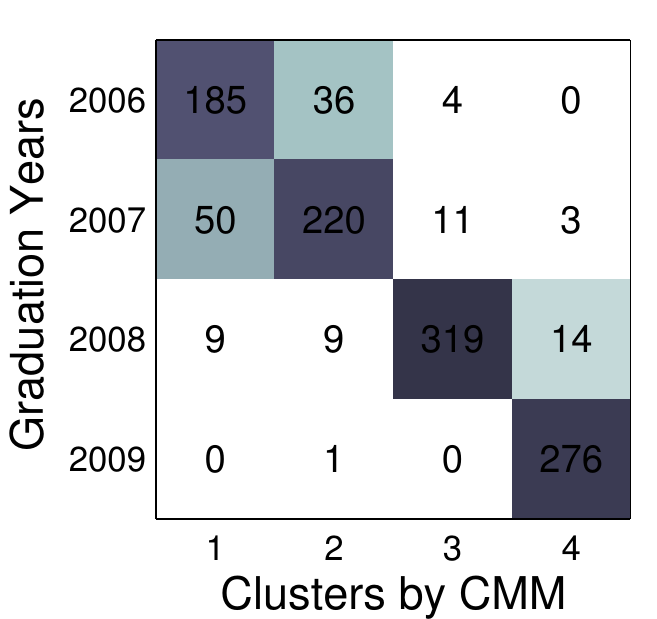}	&  \includegraphics[width=.31\columnwidth]{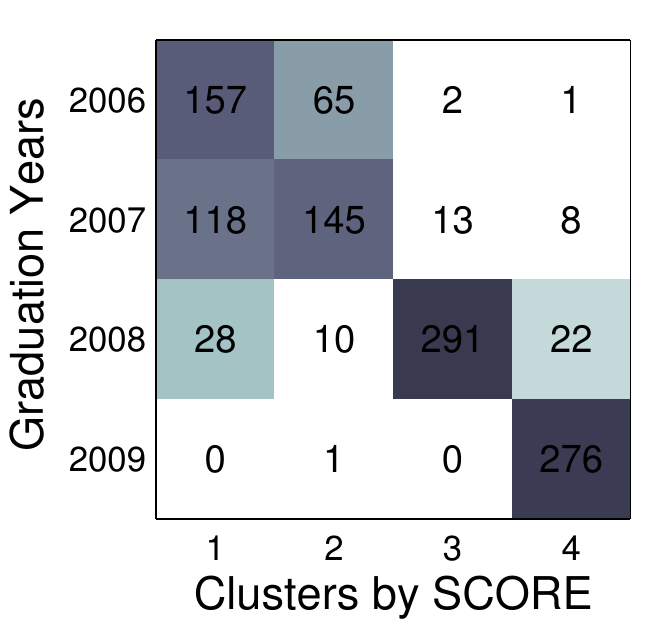}  &
\includegraphics[width=.31\columnwidth]{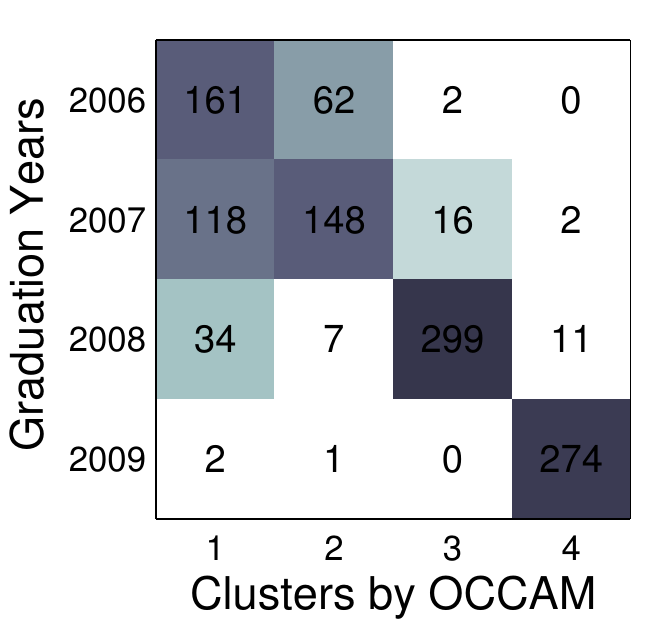} \\
\end{tabular}
\caption{The confusion matrices of CMM,  SCORE and OCCAM applied to the largest component of the Simmons College network.}
\label{fig:simmons_confmat}
\end{figure}

\subsubsection{Caltech network}

In this section, we provide experiment results on the Caltech network. This network has 769 nodes and 16656 undirected edges. 
We consider the subgraph induced by nodes with known dorm attributes, and focus on  its  largest connected component, which consists of 590 nodes and 12822 edges. The community structure is highly correlated with which of the $ 8 $ dorms a user is from, as observed in \cite{facebookdataset11,Traud2012}.

We applied CMM, SCORE and OCCAM to partition this largest component into $ r=8 $ clusters. With the dorms as the ground truth cluster labels, CMM misclassified $21.02\%$ of the nodes, whereas SCORE and OCCAM had higher misclassification rates of  $31.02\%$ and $ 32.03\% $, respectively. The confusion matrices of these methods are shown in \prettyref{fig:caltech_confmat}. We see that dorm 3 was difficult to recover and largely missed by all three methods, but our CMM algorithm better identified the other dorms.
 
\begin{figure}[H]
\centering
\begin{tabular}{ccc}
\includegraphics[width=.31\columnwidth]{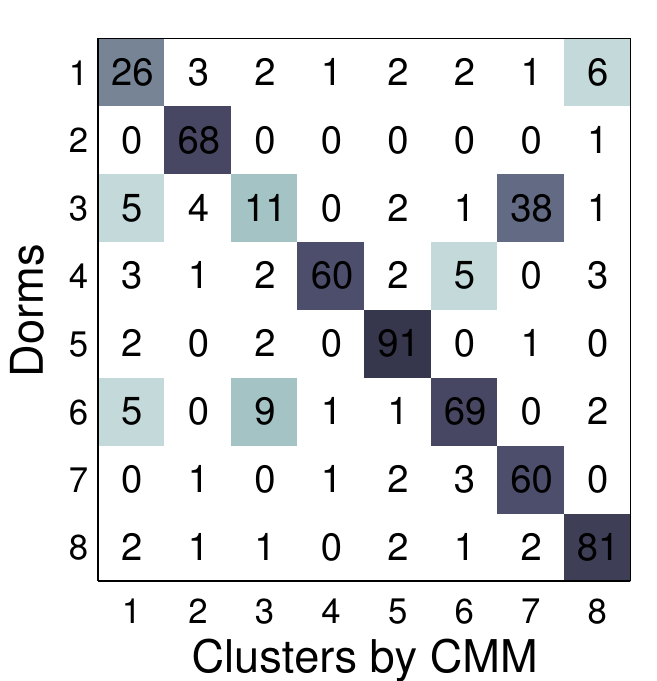}	&  \includegraphics[width=.31\columnwidth]{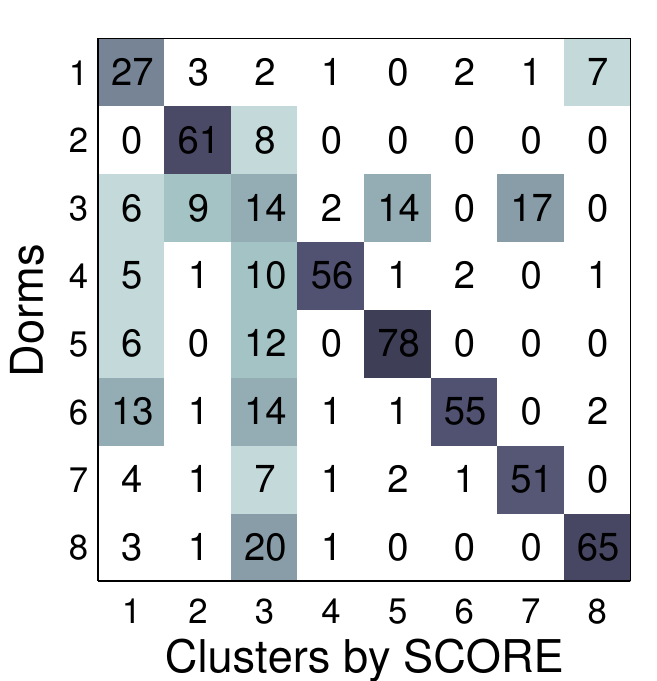}  &
\includegraphics[width=.31\columnwidth]{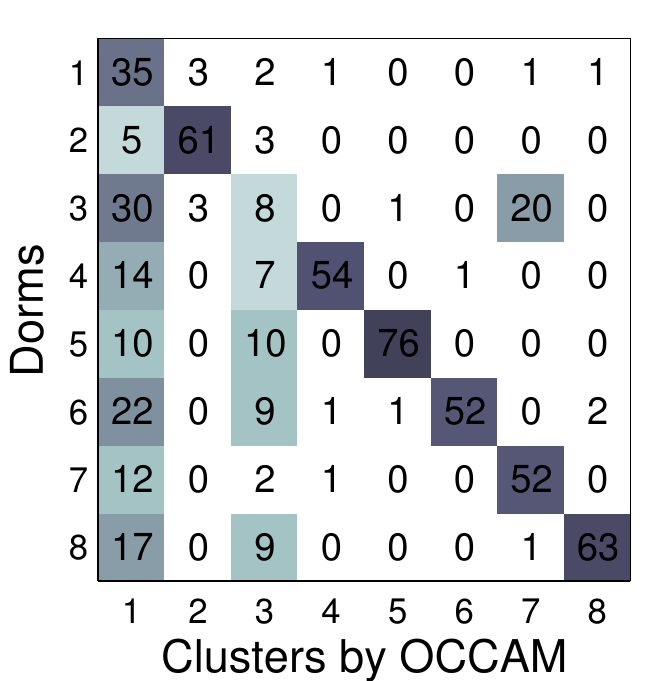} \\
\end{tabular}
\caption{The confusion matrices of CMM,  SCORE and  OCCAM applied to the largest component of the Caltech network against the dorm assignments of the users.}
\label{fig:caltech_confmat}
\end{figure}

\section{Related work}
\label{sec:related}

In this section, we discuss prior results that are related to our work.
Existing community detection methods for \DCSBM include model-based methods and spectral methods. In model-based methods such as profile likelihood and modularity maximization \citep{Newman2006modularity}, one fits the model parameters to the observed network based on the likelihood functions or modularity functions determined by the statistical structure of \DCSBM. In \cite{KN2011}, the maximum likelihood estimator is used to infer the unknown model parameters $\vct{\theta}$ and $\mtx{B}$. These estimates are then plugged into the log likelihood function, which leads to a quality function for community partitions. An estimate of the community structure is obtained by maximizing this 
quality function using a greedy heuristic algorithm. No provable theoretical guarantee is known for this greedy algorithm, and one usually needs to run the algorithm with many random initial points to achieve good performance. The work in \cite{ZLZ2012} discusses profile likelihood methods for DCSBM and the closely related modularity maximization approach. Under the assumption that the number of clusters is fixed, strong consistency is proved when the average degree is $\Omega(\log n)$, and weak consistency when it is $ \Omega(1) $. However, directly solving the
maximization problems is computationally infeasible, as it involves searching over all possible partitions.
Numerically, these optimization problems are solved heuristically using Tabu search and spectral decomposition without theoretical guarantees. The algorithm proposed in \cite{ACBL2013} involves finding an initial clustering using spectral methods, then iteratively updating the labels via maximizing conditional pseudo likelihood, which is done using the EM algorithm in each step of iteration. After simplifying the iterations into one E-step, they establish guaranteed consistency when there are only two clusters. The work in \cite{LLV14} proposes to approximate the profile likelihood functions, modularity functions or other criterions using surrogates defined in a $2$-dimensional subspace constructed by spectral dimension reduction. Thanks to the convexity of the surrogate functions, the search complexity is polynomial. The method and theory are however only applicable when there are two communities.

Spectral methods for community detection have attracted interest from diverse communities including computer science, applied math, statistics, and machine learning; see \eg \cite{RCY2011} and the references therein for results of  spectral clustering under \SBM.
The seminal work of \cite{DHM2004} on \DCSBM (proposed under the name of Extended Planted Partition model) considered a spectral method similar to that in \cite{McSherry2001}. One major drawback is that the knowledge of $\vct{\theta}$ is required in both the theory and algorithm. In the algorithm proposed in \cite{CL2009}, the adjacency matrix is first normalized by the node degrees and then thresholded entrywise, after which spectral clustering is applied. Strong consistency is proved for the setting with a fixed number of clusters. In \cite{chaudhuri}, a modified spectral clustering method was proposed using a regularized random-walk graph Laplacian, and strong consistency is established under the assumption that  the average degree grows at least as $\Omega(\sqrt{n})$. A different spectral clustering approach based on regularized graph Laplacians is considered in \cite{qin2013regularized}. Their theoretical bound on the  misclassified rates depends on the eigenvectors of the graph Laplacian, which is a still random object. Spectral clustering based on unmodified adjacency matrices and degree-normalized adjacency matrices are analyzed in \cite{LR2013} and \cite{GLM15}, which prove rigorous error rate results but did not provide numerical validation on either synthetic or real data.

It is observed in \cite{Jin2012} that spectral clustering based directly the adjacency matrix (or their normalized version) often result in inconsistent clustering in real data, such as the political blogs dataset~\cite{AG2005}, a  popular benchmark for testing community detection approaches. To address this issue, a new spectral clustering algorithm called SCORE is proposed in \cite{Jin2012}. Specifically, the second to $r$-th leading eigenvectors are divided by the first leading eigenvector elementwisely, and spectral clustering is applied to the resulting ratio matrix. In their theoretical results, an implicit assumption is that  the number of communities $ r $ is bounded by a constant, as implied by the condition (2.14) in \cite{Jin2012}. In comparison, our convexified modularity maximization approach works for growing $r$ both theoretically and empirically. As illustrated in \prettyref{sec:synthetic}, our method exhibited better performance on both the synthetic and real datasets considered there, especially when $r \geq 3$.

\section{Discussion and future work}
\label{sec:discussion}

In this paper, we studied community detection in networks with possibly highly skewed degree distributions.
We introduced a new computationally efficient methodology based on convexification of the modularity maximization formulation and a novel doubly-weighted $\ell_1$ norm $k$-median clustering procedure. Our complete algorithm runs provably in polynomial time and is computationally feasible. Non-asymptotic theoretical performance guarantees were established under \DCSBM for both approximate clustering and perfect clustering, which are consistent with the best known rate results in the literature of \SBM. 

The proposed method also enjoys good empirical performance, as was demonstrated on both synthetic data and real-world networks. On these datasets our method was observed to have performance comparable to, and sometimes better than, the state-of-the-art spectral clustering methods, particularly when there are more than two communities.

Our work involves several algorithmic and analytical novelties. We provide a tractable solution to the classical modularity maximization formulation via convexification, achieving simultaneously strong theoretical guarantees and competitive empirical performance.  The theoretical results are based on an aggregate and degree-corrected version of the density gap condition,  which is robust to a small number of outlier nodes and thus is an appropriate condition for approximate clustering.  
In our algorithms and error bounds we made use of techniques from \cite{Vershynin14, Jin2012, LR2013}, but departed from these existing works in several important aspects. In particular, we proposed a novel
$ k $-median formulation using doubly-weighted $ \ell_1 $ norms, which allows for a tight analysis that produces strong non-asymptotic guarantees on approximate recovery. Furthermore, we developed a non-asymptotic theory on perfect clustering, which is based on a divide-and-conquer primal-dual analysis and makes crucial use of certain weighted $\ell_1$ metrics that exploit the structures of \DCSBM.

A future direction important in both theory and practice, is to consider networks with overlapping communities, where a node may belong to multiple communities simultaneously. To accommodate such a setting several extensions of \SBM have been introduced in the literature. For example, \cite{ZLZ14} proposed a spectral algorithm based on the Overlapping Continuous Community Assignment Model (OCCAM). As our CMM method is shown to be an attractive alternative to spectral methods for \DCSBM, it will be interesting to extend CMM to allow for both overlapping communities and heterogeneous degrees. Another direction of interest is to develop a general theory of optimal misclassification rates for \DCSBM along the lines of \cite{Gao15,Zhangzhou15}.

\section{Proofs}
\label{sec:proofs}
In this section, we prove the theoretical results in \prettyref{thm:approx}--\ref{thm:exact}. 
Introducing the convenient shorthand $\mtx{\Theta} := \vct{\theta} \vct{\theta}^\top \in \reals_+^{n\times n}$, we can write the weighted $ \ell_1 $ norm of a matrix $ \mtx{Z} $ in \prettyref{def:weighted_norm} as
\[
\|\mtx{Z}\|_{1, \vct{\theta}} =\sum_{1 \leq i, j \leq n} |\theta_i Z_{ij} \theta_j| = \|\mtx{\Theta} \circ \mtx{Z}\|_1, 
\]
where $ \circ $ denotes the Hadamard (element-wise) product.  Several standard matrix norms will also be used: the spectral norm $\|\mtx{Z}\|$ (the largest singular value of $ \mtx{Z} $); the nuclear norm  $\|\mtx{Z}\|_\ast$ (the sum of the singular values); the $ \ell_1 $ norm $\|\mtx{Z}\|_1=\sum_{i,j}|Z_{ij}|$; the $\ell_\infty$ norm $\|\mtx{Z}\|_\infty=\max_{i,j} |Z_{ij}|$; 
and the $\ell_{\infty} \to \ell_1$ operator norm $\|\mtx{Z}\|_{\infty \to 1} = \sup_{ \| \vct{v} \|_\infty \le 1} \| \mtx{Z} \vct{v} \|_1$.  

For any vector $\vct{v} \in \mathbb{R}^n$, we denote by $\diag(\vct{v})$ the $n\times n$ diagonal matrix whose diagonal entries are correspondingly the entries of $\vct{v}$. For any matrix $\mtx{M} \in \reals^n$, let $\diag(\mtx{M})$ denote the $n \times n$ diagonal matrix with diagonal entries given by the corresponding diagonal entries of $\mtx{M}$.  We denote absolute constants by $C, C_0, c_1$, etc, whose value may change line by line. \\

Recall that $ \vct{d} $ and $ \vct{f} $ are the vectors of node degrees and their expectations, respectively, where 
\[
f_i = \theta_i H_a,
\]
for each $i \in C^\ast_a$, $a=1, \ldots, r$. 
A key step in our proofs is to appropriately control the deviation of the degrees from their expectation. This is done in the following lemma.
\begin{lemma}
\label{lmm:fd}
Under \DCSBM, with probability at least $0.99$, we have
\[
\max \left(\|\vct{f}\vct{f}^\top - \vct{d} \vct{d}^\top\|_{\infty \to 1}, \|\vct{f}\vct{f}^\top - \vct{d} \vct{d}^\top\|_1 \right) \leq C \big(n + \sqrt{n \|\vct{f}\|_1} \, \big) \|\vct{f}\|_1
\]
for some absolute constant $C>0$.
\end{lemma}
\begin{proof}
Since
\[
\vct{f}\vct{f}^\top - \vct{d} \vct{d}^\top  = \vct{f}(\vct{f}-\vct{d})^\top + (\vct{f} - \vct{d})\vct{d}^\top,
\]
we have
\begin{align*}
 \| \vct{f}\vct{f}^\top - \vct{d} \vct{d}^\top \|_{\infty \to 1} &\leq   \|\vct{f}(\vct{f}-\vct{d})^\top\|_{\infty \to 1} + \|(\vct{f} - \vct{d})\vct{d}^\top \|_{\infty \to 1}
 \\
 & = \|\vct{f}\|_1\|\vct{f}-\vct{d}\|_1 + \|\vct{f}-\vct{d}\|_1\|\vct{d}\|_1
 \\
 & = \|\vct{f}-\vct{d}\|_1(\|\vct{f}\|_1 + \|\vct{d}\|_1), 
\end{align*}
and
\begin{align*}
 \| \vct{f}\vct{f}^\top - \vct{d} \vct{d}^\top \|_{1} &\leq   \|\vct{f}(\vct{f}-\vct{d})^\top\|_{1} + \|(\vct{f} - \vct{d})\vct{d}^\top \|_{1}
 \\
 & = \|\vct{f}\|_1\|\vct{f}-\vct{d}\|_1 + \|\vct{f}-\vct{d}\|_1\|\vct{d}\|_1
 \\
 & = \|\vct{f}-\vct{d}\|_1(\|\vct{f}\|_1 + \|\vct{d}\|_1).
\end{align*}
We bound $\| \vct{f} - \vct{d} \|_1$ and $ \| \vct{f} \|_1 $ separately.  For each $i \in C_a^*$, there holds $\E d_i = f_i - \theta_i^2 B_{aa}$ and
\[
\Var(d_i) = \sum_{j=1}^n \Var(A_{ij}) \leq \sum_{j=1}^n \E(A_{ij}) = \E(d_i) = f_i - \theta_i^2B_{aa} \leq f_i.
\]
Therefore, we have
\[
\E \left|f_i - \theta_i^2 B_{aa} - d_i\right| \leq \sqrt{\E  \left|f_i - \theta_i^2 B_{aa} - d_i\right|^2} = \sqrt{\Var(d_i)} \leq \sqrt{f_i},
\]
which implies that $\E\left| f_i - d_i \right| \leq 1+ \sqrt{f_i}$ and
\[
\E \sum_{i=1}^n \left| f_i - d_i \right| \leq n + \sum_{i=1}^n \sqrt{f_i} \leq n + \sqrt{n\|\vct{f}\|_1}.
\]
By Markov's inequality, with probability $0.995$, there holds
\[
\|\vct{f} - \vct{d} \|_1 \leq C(n + \sqrt{n \|\vct{f}\|_1})
\]
for an absolute constant $C$.
To bound $ \| \vct{d} \|_1  $, we observe that since $\E d_i \leq f_i$ and $d_i \ge 0$, there holds $\E \|\vct{d}\|_1 \leq \|\vct{f}\|_1$. By Markov's inequality, with probability at least $0.995$, there holds $\|\vct{d}\|_1 \leq C\|\vct{f}\|_1$ for some absolute constant. Combining these bounds on $ \|\vct{f} - \vct{d} \|_1 $ and $ \| \vct{d} \|_1  $ proves \prettyref{lmm:fd}.
\end{proof}

\subsection{Proof of \prettyref{thm:approx}}
\label{sec:proof_approx}

Recall that the vector $\vct{f} \in \mathbb{R}^n$ is defined by letting $f_i = \theta_i H_a$  for $i \in C_a^*$,
where  $H_a$ is defined in \prettyref{eq:GandH}. It follows from the optimality of $\Yhat$ that
\begin{align*}
0 & \le \langle  \Yhat -\mtx{Y}^\ast , \mtx{A} -\lambda \vct{d}\vct{d}^\top \rangle
\\
& = \underbrace{\langle  \Yhat -\mtx{Y}^\ast , \E \mtx{A}  -\lambda \vct{f}\vct{f}^\top \rangle}_{S_1}
+ \underbrace{\lambda \langle  \Yhat -\mtx{Y}^\ast ,   \vct{f}\vct{f}^\top - \vct{d}\vct{d}^\top \rangle}_{S_2}
+ \underbrace{\langle  \Yhat -\mtx{Y}^\ast ,  \mtx{A} - \E \mtx{A} \rangle}_{S_3} .
\end{align*}
We control the terms $ S_1 $, $ S_2 $ and $ S_3 $ separately below.

\paragraph{Upper bound for $S_1$} 
For each pair $i, j \in C_a^*$ and $i \neq j$, we have $\widehat{Y}_{ij} - Y^*_{ij} \leq 0$, $\E(A_{ij})=\theta_i \theta_j B_{aa}$, and $f_i f_j = \theta_i \theta_j H_aH_b$. Hence the condition \prettyref{eq:RDGC} implies that $\E(A_{ij}) - \lambda f_i f_j \geq \delta \theta_i \theta_j$, whence
\[
(\widehat{Y}_{ij} - Y^*_{ij}) (\E(A_{ij}) - \lambda f_i f_j ) \leq -\delta \theta_i \theta_j |\widehat{Y}_{ij} - Y^*_{ij}|.
\]
Similarly, for each pair $i \in C_a^*$ and $j \in C_b^*$ with  $1\leq a < b \leq r$, we have
\[
(\widehat{Y}_{ij} - Y^*_{ij}) (\E(A_{ij}) - \lambda f_i f_j ) \leq -\delta \theta_i \theta_j |\widehat{Y}_{ij} - Y^*_{ij}|.
\]
Combining the last two inequalities, we obtain the bound
\begin{equation*}
 S_1 :=\langle  \Yhat -\mtx{Y}^\ast , \E \mtx{A}  -\lambda \vct{f}\vct{f}^\top \rangle \leq   - \delta  \|\Ystar - \widehat{ \mtx{Y} } \|_{1, \vct{\theta}}.
\end{equation*}

\paragraph{Upper bound for $S_2$} 
By Grothendieck's inequality~\citep{Grothendieck, Lindenstrauss} we have
\begin{align*}
 \langle  \Yhat -\mtx{Y}^\ast ,  \vct{f}\vct{f}^\top - \vct{d} \vct{d}^\top \rangle &  \le 2 \sup_{\mtx{Y} \succeq 0, \diag(\mtx{Y}) = \mtx{I}  } \big|  \langle \mtx{Y} ,  \vct{f}\vct{f}^\top - \vct{d} \vct{d}^\top \rangle \big| \\
 & \leq 2 K_G \| \vct{f}\vct{f}^\top - \vct{d} \vct{d}^\top \|_{\infty \to 1},
\end{align*}
where $K_G$ is Grothendieck's constant known to satisfy $K_G \le 1.783$.
Since $\lambda \leq \min_{1 \leq a \leq r} \frac{B_{aa}}{H_a^2}$, applying \prettyref{lmm:fd} on  $ \| \vct{f}\vct{f}^\top - \vct{d} \vct{d}^\top \|_{\infty \to 1} $ ensures that with probability at least $0.99$,
\[
S_2 \leq C\left(\min_{1 \leq a \leq r} \frac{B_{aa}}{H_a^2}\right)  \|\vct{f}\|_1 \left(  \sqrt{n\|\vct{f}\|_1 }  + n \right)
\]
for some absolute constant $C$.

\paragraph{Upper bound for $S_3$}  
Observe that
\begin{align*}
 \langle  \Yhat -\mtx{Y}^\ast ,  \mtx{A} - \E \mtx{A} \rangle  \le 2 \sup_{\mtx{Y} \succeq 0, \diag \mtx{Y} = \mtx{I}  } \big|  \langle \mtx{Y} ,  \mtx{A} - \E \mtx{A} \rangle \big|.
\end{align*}
It follows from Grothendieck's inequality that
\begin{align*}
 \sup_{\mtx{Y} \succeq 0, \diag(\mtx{Y}) = \mtx{I}  } \big|  \langle \mtx{Y} ,  \mtx{A} - \E \mtx{A} \rangle \big| \le K_G  \| \mtx{A} - \E \mtx{A} \|_{\infty \to 1} .
\end{align*}
The norm on the last RHS can be expressed as
\begin{align*}
 \| \mtx{A} - \E \mtx{A} \|_{\infty \to 1} = \sup_{\vct{x} : \| \vct{x} \|_\infty \le 1} \| ( \mtx{A} - \E \mtx{A} ) \vct{x} \|_1 = \sup_{\vct{x}, \vct{y} \in \{\pm 1\}^n } | \vct{x}^\top ( \mtx{A} - \E \mtx{A} ) \vct{y} |.
\end{align*}
For each fixed pair of sign vectors $\vct{x}, \vct{y} \in \{ \pm 1 \}^n$, Bernstein's inequality ensures that for each $t >0$, with probability at most $2e^{-t}$ there holds the inequality
\begin{align*}
| \vct{x}^\top ( \mtx{A} - \E \mtx{A} ) \vct{y} | \ge \sqrt{ 8 t \sigma^2 } + \frac{4}{3} t ,
\end{align*}
where $
\sigma^2 := \sum_{i<j} \var A_{ij}  \le \frac{1}{2} \sum_{a, b=1}^{r} B_{ab} G_a G_b = \frac{1}{2} \| \vct{f} \|_1 .$
Setting $t=2n$ and applying the union bound over all sign vectors, we obtain that with probability at most $2(e/2)^{-2n}$,
\begin{align*}
 \| \mtx{A} - \E \mtx{A} \|_{\infty \to 1}  \ge \sqrt{ 8 n \| \vct{f} \|_1 } + \frac{8}{3} n.
\end{align*}
It follows that with probability at least $1- 2(e/2)^{-2n}$,
$$
S_3 \le 2 K_G  \sqrt{ 8 n \| \vct{f} \|_1 } + \frac{16 K_G}{3} n.
$$

Putting together the bounds for $ S_1 $, $ S_2 $ and $ S_3 $, we conclude that with probability at least $0.99- 2 (e/2)^{-2n}$,
the bound \prettyref{eq:pillar1} holds.

\subsection{Proof of \prettyref{thm:kmedian}}
\label{sec:proof_kmedian}

Recall that $(\overline{\mtx{\Psi}}, \overline{\mtx{X}})$ is the exact optimal solution to the weighted $ k $-median problem~\prettyref{eq:k-median}, $(\PsiCheck, \XCheck)$ is the approximate solution, and $\widehat{\mtx{W}}:= \Yhat\mtx{D}$ is the column-weighted version of the solution $ \Yhat $ to the convex program~\prettyref{eq:cvx}. The last constraint in \prettyref{eq:k-median} ensures that the row vectors of $\overline{\mtx{X}}$ and $\XCheck$ are subsets of the row vectors of $\widehat{\mtx{W}}$. If we  define the matrices $\overline{\mtx{W} }:=\overline{\mtx{\Psi}} \; \overline{\mtx{X}}$ and $\widecheck{\mtx{W}} :=\PsiCheck \; \XCheck$, then the row vectors of $\overline{\mtx{W} }$ and $\widecheck{\mtx{W}}$ are also subsets of the row vectors of $\widehat{\mtx{W}}$. For any matrix $\mtx{M}$, we let $\mtx{M}_{i \bullet}$ denote the $i$-th row vector of $\mtx{M}$, and $\mtx{M}_{ \bullet j}$ the $j$-th column vector of $\mtx{M}$. 
At a high level, we prove \prettyref{thm:kmedian} by translating 
the upper bound on the weighted error $ \| \widehat{\mtx Y} - \Ystar \|_{1, \vct{\theta}} $, given in~\prettyref{eq:pillar2} in~\prettyref{cor:Fmodel},
to an upper bound on the weighted misclassification rate defined in \prettyref{def:calE}.
This is done in three steps.

\paragraph*{Step 1}

As shown in Section \ref{sec:kmedian}, the true partition matrix admits the decomposition $\Ystar= \mtx{\Psi}^\ast ( \mtx{\Psi}^\ast)^\top$, where $\mtx{\Psi}^\ast \in \mathbb{M}_{n, r}$ is the true membership matrix.
Letting $\mtx{W}^*:=\Ystar\mtx{D} \in \mathbb{R}^{n\times n}$ and $\mtx{X}^* := (\mtx{\Psi}^\ast)^\top \mtx{D} \in \mathbb{R}^{r \times n},$  we have the expression $\mtx{W}=\mtx{\Psi}^\ast \mtx{X}^*$.
We now define a matrix $\widetilde{\mtx{X}} \in \mathbb{R}^{r \times n}$  by setting its $ k $-th row to 
\[
\widetilde{\vct{X}}_{k \bullet}:=\arg \min\limits_{\vct{x} \in \{\widehat{\mtx{W}}_{i \bullet}: i \in C_k^*\}} \|\vct{x} - \vct{X}^*_{k \bullet }\|_1,
\quad \text{for each~} k=1, \ldots, r.
\]
Note that $ \widetilde{\vct{X}} $ also satisfies $\text{Rows}(\widetilde{\mtx{X}}) \subseteq \text{Rows}(\widehat{\mtx{W}})$.
Set $\widetilde{\mtx{W}} := \mtx{\Psi}^\ast \widetilde{\mtx{X}} \in \mathbb{R}^{n \times n}$. By definition we have the inequality
\begin{align*}
\|\mtx{D}(\widehat{\mtx{W}} - \mtx{W}^* )\|_1
&= \sum_{k=1}^r \sum_{i \in C_k^*} d_i \|\widehat{\vct{W}}_{i \bullet } -\vct{X}^*_{k \bullet}\|_1 \\
&\geq \sum_{k=1}^r \sum_{i \in C_k^*} d_i \|\widetilde{\vct{X}}_{k \bullet } - \vct{X}^*_{k \bullet}\|_1 
= \|\mtx{D}(\widetilde{\mtx{W}} - \mtx{W}^* ) \|_1,
\end{align*}
which implies that $\|\mtx{D}(\widetilde{\mtx{W}} - \widehat{\mtx{W}})\|_1 \leq \|\mtx{D}(\widehat{\mtx{W}} - \mtx{W}^* ) \|_1 + \|\mtx{D}(\widetilde{\mtx{W}} - \mtx{W}^* )\|_1 \leq 2\|\mtx{D}(\widehat{\mtx{W}} - \mtx{W}^*)\|_1$. Since $(\mtx{\Psi}^\ast, \widetilde{\mtx{X}})$ is feasible to the optimization \prettyref{eq:k-median}, we have
\[
\|\mtx{D}(\overline{\mtx{W}} - \widehat{\mtx{W}})\|_1 \leq \|\mtx{D}(\widetilde{\mtx{W}} - \widehat{\mtx{W}})\|_1\leq 2\|\mtx{D}(\widehat{\mtx{W}} - \mtx{W}^* )\|_1,
\]
whence
\[
\|\mtx{D}( \widecheck{\mtx{W}}  - \widehat{\mtx{W}})\|_1 \le \frac{20}{3} \|\mtx{D}( \overline{ \mtx{W} } - \widehat{\mtx{W}})\|_1 \leq \frac{40}{3} \|\mtx{D}(\widehat{\mtx{W}} - \mtx{W}^* )\|_1.
\]
Putting together, we obtain that
\begin{align}
\|\mtx{D}(\widecheck{\mtx{W}} - \mtx{W}^* )\|_1 
\leq \|\mtx{D}(\widehat{\mtx{W}} - \mtx{W}^* )\|_1 + \|\mtx{D}(\widecheck{\mtx{W}} - \widehat{\mtx{W}})\|_1
\leq \frac{43}{3}\|\mtx{D}(\widehat{\mtx{W}} - \mtx{W}^* )\|_1.
\label{eq:DWbound}
\end{align}
Define a matrix $\widecheck{\mtx{Y}} \in \reals^{n \times n} $ by
\begin{align}
\widecheck{Y}_{ i j} =
\begin{cases}
  \widecheck{W}_{i j}/d_j,   & \text{ if } d_j>0, \\
 0,  &  \text{ if }  d_j=0. 
\end{cases}
\label{eq:defcheckY}
\end{align}
For each $j \in [n]$, if $ d_j>0$, then it follows from the above definition
that $\widecheck{W}_{ij}=\widecheck{Y}_{ij}d_j$. Suppose $d_j=0$; because
 $\text{Rows} ( \widecheck{\mtx{W}} ) \subseteq \text{Rows} ( \widehat{\mtx{W}} )$ and
$\widehat{\mtx{W}}\mtx = \Yhat \mtx{D}$,  for each $i \in [n]$, there exists an index $i' \in [n]$ such that 
$$
\widecheck{W}_{ij}=\hat{W}_{i' j} =\hat{Y}_{i'j} d_j =0 = \widecheck{Y}_{ij} d_j. 
$$
Putting together, we conclude that $\widecheck{\mtx{W}}=\widecheck{\mtx{Y} }\mtx{D}$. 
In view of the bound~\prettyref{eq:DWbound} and the definitions $\widehat{\mtx{W}} :=\Yhat \mtx{D}$ and $\mtx{W}^* :=\Ystar \mtx{D}$, 
we get that
\begin{equation}
\label{eq:checkhat}
\begin{aligned}
\| \widecheck{\mtx Y} - \Ystar \|_{1, \vct{d}} 
& =\| \mtx{D} ( \widecheck{\mtx Y} - \Ystar ) \mtx{D} \|_1 \\ &=\|\mtx{D}(\widecheck{\mtx{W}} - \mtx{W}^* )\|_1 \\
& \le \frac{43}{3}\|\mtx{D}(\widehat{\mtx{W}} - \mtx{W}^* )\|_1 \\
&= \frac{43}{3} \|  \mtx{D} ( \Yhat - \Ystar ) \mtx{D}  \|_1  
  =\frac{43}{3} \| \Yhat - \Ystar \|_{1, \vct{d}},
\end{aligned}
\end{equation}
where the weighted $ \ell_1 $ norm $ \| \cdot \|_{1, \vct{d}} $ is defined analogously to $ \| \cdot \|_{1, \vct{\theta}}  $ in~\prettyref{def:weighted_norm}. 

\paragraph*{Step 2}

The bound in~\prettyref{eq:checkhat} is weighted by the empirical degrees $ \vct{d} $. Our next step is to convert this bound into one that is weighted by the population quantity $ \vct{f} $. 
Recall that $\text{Rows} ( \widecheck{\mtx{W}} ) \subseteq \text{Rows} ( \widehat{\mtx{W}} )$ and
$\widehat{\mtx{W}}\mtx = \Yhat \mtx{D}$. If $d_j>0$, then for any $i \in [n]$, there exists an $i'$ such that 
$$
\widecheck{Y}_{ij} = \widecheck{W}_{ij}/d_j = \widehat{W}_{i' j} /d_j = \widehat{Y}_{i' j}.
$$
Since $\Yhat $ is feasible to the convex relaxation~\prettyref{eq:cvx}, we have $\mtx{0} \leq \Yhat \leq \mtx{J}$. It follows that $\mtx{0} \leq \widecheck{\mtx{Y}} \leq \mtx{J}$ and hence $\| \widecheck{\mtx{Y}} - \Ystar \|_{\infty} \leq 1$. 
Setting $M:= \|\vct{f}\vct{f}^T - \vct{d}\vct{d}^T\|_1$, we observe that any matrix $\mtx{Z}$ satisfies the bound
\[
\left|\|\mtx{Z}\|_{1, \vct{f}} - \|\mtx{Z}\|_{1, \vct{d}}\right| \leq \|\mtx{Z} \circ (\vct{f}\vct{f}^T - \vct{d}\vct{d}^T)\|_1 \leq M \|\mtx{Z}\|_\infty,
\]
where $ \|\mtx{Z}\|_{1, \vct{f}} $ and $\|\mtx{Z}\|_{1, \vct{d}}$ are defined in the same fashion as $\|\mtx{Z}\|_{1, \vct{\theta}}$ given in \prettyref{def:weighted_norm}. 
Therefore, the bound~\prettyref{eq:checkhat} implies that
\begin{align}
\| \widecheck{\mtx Y} - \Ystar \|_{1, \vct{\theta}} &= \frac{1}{h^2}\| \widecheck{\mtx Y} - \Ystar \|_{1, \vct{f}}  \le \frac{1}{h^2}\left(\| \widecheck{\mtx Y} - \Ystar \|_{1, \vct{d}} +M\right)
\nonumber \\
&\lesssim \frac{1}{h^2}\left(\| \widehat{\mtx Y} - \Ystar \|_{1, \vct{d}} +M\right) \leq \frac{1}{h^2}\left(\| \widehat{\mtx Y} - \Ystar \|_{1, \vct{f}} +2M\right)
\nonumber \\
& \lesssim \| \widehat{\mtx Y} - \Ystar \|_{1, \vct{\theta}} + \frac{M}{h^2}.
\label{eq:perturb_bound}
\end{align}
To bound the second term above, we apply \prettyref{lmm:fd} to get that with probability at least 0.99,
\begin{equation*}
\label{eq:Mbound}
M \leq C \|\vct{f}\|_1 (\sqrt{n \|\vct{f}\|_1} + n).
\end{equation*}
Also note that under the $\mathcal{F}(n, r, p, q, g)$-model,
\begin{equation*}
\label{eq:h}
H_1 = \cdots = H_r = h: = (p+(r-1)q)g ,
\end{equation*}
 which implies that $f_i = \theta_i h, \forall i \in [n]$ and
\begin{equation*}
\label{eq:fl1}
 \|\vct{f}\|_1 = r(p+(r-1)q)g^2 .
\end{equation*}
Combining the last three equations gives the following bound on the second term of~\prettyref{eq:perturb_bound}:
\begin{align}
\frac{M}{h^2} \lesssim \frac{r ( \sqrt{ n \| \vct{f} \|_1} +n ) } {p+ (r-1)q} \le \frac{r ( r g \sqrt{np} + n ) }{ \delta}.
\label{eq:Mh2bound}
\end{align}
We can control the first term in~\prettyref{eq:perturb_bound} using the bound~\prettyref{eq:pillar2} in~\prettyref{cor:Fmodel}. 
Putting together, straightforward calculation yields the inequality
\begin{equation}
\label{eq:pillar3}
\| \widecheck{\mtx Y} - \Ystar \|_{1, \vct{\theta}} \lesssim \frac{1}{\delta} r (n + rg\sqrt{np}).
\end{equation}

\paragraph*{Step 3}
Recall that $\diag(\vct{\theta})$ is the $n\times n$ diagonal matrix whose diagonal entries are correspondingly the entries of $\vct{\theta}$.
For each $a=1,\ldots,r$, define the set of node indices
\[
S_{a} := \left\{ i\in C_{a}^{*}:\|(\widecheck{\vct{Y}}_{i \bullet }-\vct{Y}_{i \bullet}^{*}) \diag(\vct{\theta}) \|_{1}\geq g \right\} ,
\]
and let $S := \bigcup_{a=1}^{r}S_{a}$. It follows from~\prettyref{eq:pillar3} that
\begin{equation}
\label{eq:Sbound0}
\begin{aligned}
\sum_{i\in S}\theta_{i} &\le\sum_{i=1}^{n}\frac{\theta_{i}}{g}\|(\widecheck{\vct{Y}}_{i \bullet }-\vct{Y}_{i \bullet}^{*})\diag(\vct{\theta}) \|_{1} \\
&= \frac{1}{g}\|\widecheck{\vct{Y}}-\mtx{Y}^{*}\|_{1, \vct{\theta}} \lesssim \frac{1}{\delta g} r (n + rg\sqrt{np}).
\end{aligned}
\end{equation}

Consider the set $T_{a}:=C_{a}^{*} \backslash S_{a}$ for each $a = 1, \ldots, r$. There are three cases for each $T_a$. In the first case, $T_a = \emptyset$, and we denote by $R_1$ the collection of all such indices $a$. 
In the second case, $T_a \neq \emptyset$ and  $\vct{\widecheck{\Psi}}_{i \bullet} = \vct{\widecheck{\Psi}}_{j \bullet}$
for all $i, j \in T_a$. 
We say that these $T_a$'s are pure, and denote by $R_2$ the collection of all such indices $ a $. Finally, we set $R_3 := \{1, \ldots, r\} \backslash (R_1 \cup R_2)$; for each $a \in R_3$, we say that $T_a$ is impure since there exist $ i,j \in T_a $ such that $\vct{\widecheck{\Psi}}_{i \bullet} \neq \vct{\widecheck{\Psi}}_{j \bullet}$.

For each $a \in R_1$, we have $S_a = C_a^*$, which implies that
\begin{equation}
\label{eq:R1_upper}
\sum_{i \in S} \theta_i \geq \sum_{i \in \bigcup_{a \in R_1} C_a^*} \theta_i = |R_1|g.
\end{equation}
For each pair $a, b \in R_2 \cup R_3$ with $a\neq b$, by definition we know that $T_a \neq \emptyset$ and $T_b \neq \emptyset$. Then for each pair $i\in T_{a} \subseteq C_a^*$, $j\in T_{b} \subseteq C_b^*$, we have
\begin{align*}
&\|\vct{\widecheck{Y}}_{i \bullet }\diag(\vct{\theta})-\vct{\widecheck{Y}}_{j \bullet }\diag(\vct{\theta})\|_{1} 
\\
& \geq\|(\vct{Y}_{i \bullet}^{*}-\vct{Y}_{j \bullet}^{*})\diag(\vct{\theta}) \|_{1}-\|(\vct{Y}_{i \bullet}^{*}-\vct{\widecheck{Y}}_{i\bullet})\diag(\vct{\theta}) \|_{1}-\|(\vct{Y}_{j\bullet}^{*}-\vct{\widecheck{Y}}_{j\bullet})\diag(\vct{\theta}) \|_{1}\\
 & >2g-g-g = 0,
\end{align*}
whence $\vct{\widecheck{Y}}_{i \bullet}  \neq \vct{\widecheck{Y}}_{j \bullet}$. 
We claim that this implies $\vct{\widecheck{\Psi}}_{i \bullet}  \neq \vct{\widecheck{\Psi}}_{j \bullet}$. 
Suppose that this claim is not true.  For each $k \in [n]$,  if $d_k=0$, then $\widecheck{Y}_{ik}= \widecheck{Y}_{jk}=0$ in view of the definition of $\widecheck{\mtx{Y}}$ in \prettyref{eq:defcheckY};
if $d_k>0$, then the definition $\widecheck{\mtx{W}} = \widecheck{\mtx{\Psi}} \widecheck{\mtx{X}}$ implies that
$$
\widecheck{Y}_{ik} = \frac{1}{d_k} \widecheck{W}_{ik} = 
\frac{1}{d_k}  \Iprod{\vct{\widecheck{\Psi}}_{i \bullet} }{ \vct{ \widecheck{X}_{\bullet k} } }
=\frac{1}{d_k}  \Iprod{\vct{\widecheck{\Psi}}_{j \bullet} }{ \vct{ \widecheck{X}_{\bullet k} } } =\frac{1}{d_k} \widecheck{W}_{jk} =
\widecheck{Y}_{jk}.
$$
Therefore, we have $\vct{\widecheck{Y}}_{i \bullet}  = \vct{\widecheck{Y}}_{j \bullet}$, which is a contradiction.

In conclusion, we have proved that for each pair $a, b \in R_2 \cup R_3$ with $a \neq b$ and each pair $i\in T_{a}$, $j\in T_{b}$, we have $\vct{\widecheck{\Psi}}_{i \bullet}  \neq \vct{\widecheck{\Psi}}_{j \bullet}$. 
Moreover, since for each $a \in R_2$, the set $ T_a $ is pure  by definition,  there exists a permutation matrix $\mtx{\Pi} \in \mathcal{S}_r$ such that for all $i \in \bigcup_{a \in R_2} T_a$, there holds $(\widecheck{\mtx{\Psi}} \mtx{\Pi})_{i \bullet} = \mtx{\Psi}^*_{i \bullet}$. Recalling~\prettyref{def:calE}, we conclude that the set $(\bigcup_{a \in R_3} T_a) \bigcup S$ contains the misclassified node set with respect to $ \mtx{\Pi} $. It follows that
\begin{equation}
\label{eq:mis_upper}
 \textstyle{\sum_{i \in \mathcal{E}(\mtx{\Pi})} \theta_{i} } \leq \sum_{i \in S} \theta_i + \sum_{i \in \bigcup_{a \in R_3} T_a} \theta_i \leq \sum_{i \in S} \theta_i + |R_3| g.
\end{equation}

The matrix $\PsiCheck$ consists of at most $r$ distinct row vectors. Because $R_2$ is pure and $R_3$ is impure by definition, we have the inequality
\[
|R_2| + 2|R_3| \leq r = |R_1| + |R_2| + |R_3|,
\]
which implies that 
\begin{equation}
\label{eq:R1_R3}
|R_3| \leq |R_1|.
\end{equation}

Applying the bounds~\prettyref{eq:mis_upper}, \prettyref{eq:R1_R3}, \prettyref{eq:R1_upper} and~\prettyref{eq:Sbound0} in order, we obtain
\begin{align*}
 \textstyle{ \sum_{i \in \mathcal{E}(\mtx{\Pi})} } \theta_{i} &\le 2\sum_{i\in S}\theta_{i}  \lesssim \frac{1}{\delta g} r (n + rg\sqrt{np}),
\end{align*}
thereby proving the inequality~\prettyref{eq:Sbound} in the theorem.

\subsection{Proof of \prettyref{thm:exact}}

Let us first recall and introduce some key notations and definitions. For each $k \in [r]$, define $\vct{\theta}^{(a)} \in \reals_+^n$  such that $\theta^{(a)}_i=\theta_i$ if $i \in C^\ast_a$ and $ \theta^{(a)}_i=0$ otherwise.  
For any vector $\vct{v} \in \reals^n$ and $1 \le a \le r$, we let $\vct{v}_{(a)} \in \reals^{\ell_a}$ 
denote the restriction of $\vct{v}$ to entries in $C^*_a$. 
For any matrix $\mtx{M} \geq \mtx{0}$,  let $\mtx{M}^{\frac{1}{2}}$ denote the matrix such that its $(i,j)$-entry is $\sqrt{M_{ij}}$. Similarly, for any matrix $\mtx{M} > \mtx{0}$, let $\mtx{M}^{-\frac{1}{2}}$ denote the matrix such that its $(i,j)$-entry is $\frac{1}{\sqrt{M_{ij}}}$.

Let $G:=\sum_{i=1}^n \theta_i = \|\vct{\theta}\|_1$ and we have for all $1 \leq a \leq r$,
\[
H_a = \sum_{k=1}^r B_{ak}G_k= qG + (p - q)G_a.
\]
A simple implication is
\[
H_a \geq qG_a + (p - q)G_a=pG_a \geq p G_{\min}.
\]
By the assumption \prettyref{eq:assump} and the fact that $\delta < p,$
\begin{align}
\label{eq:lower_sim}
C_0^2 \frac{\log n}{G_{\min}\theta_{\min}}\leq \frac{\delta^2}{p} < \delta.
\end{align}
~\\

We now turn to the proof of \prettyref{thm:exact}. 
In the proof we make use of several technical lemmas given in the Appendix.
\begin{proof}
Recall that
$
f_i = \theta_i H_a,
$
for each $i \in C^\ast_a$, $a=1, \ldots, r$. 
The following lemma, complementing \prettyref{lmm:fd}, characterizes the relationship between the degrees $d_1, \ldots, d_n$ and the population quantities $f_1, \ldots, f_n$.

\begin{lemma}
\label{lmm:degrees}
With probability at least $1 - \frac{2}{n^2}$, for all $i=1, \ldots, n$,
	\begin{equation}
	\label{eq:d_bound1}
	|d_i - f_i|\leq \max(\sqrt{12 f_i \log n}, 4 \log n + 1).
	\end{equation}
If we further assume that the condition~\eqref{eq:assump} holds with some large enough numerical constant $C_0$, there holds for all $i=1, \ldots, n$
	\begin{equation}
	\label{eq:d_bound2}
	|d_i - f_i|\leq \frac{\delta}{5p}f_i.
	\end{equation}
\end{lemma}
\noindent We prove this lemma in \prettyref{sec:proof_lmm_degrees} to follow.

Back to the proof of \prettyref{thm:exact},
we assume without loss of generality that the nodes in the same cluster have adjacent indices. Recall that for $1\le a \le r$, $\ell_a$ is the size of community $a$. Then we have
\begin{equation}
\label{eq:output}
\Ystar=
\begin{bmatrix}
\mtx{J}_{l_1} &~ &~
\\
~ &\ddots & ~
\\
~ & ~ & \mtx{J}_{l_r}
\end{bmatrix} \in \Pnr.
\end{equation}
Recall the decomposition $\Ystar=\mtx{\Psi}^\ast (\mtx{\Psi}^\ast)^\top$, where
\[
\mtx{\Psi}^\ast:=[\vct{v}_1, \ldots, \vct{v}_r]:=
\begin{bmatrix}
\vct{1}_{l_1} & \vct{0} & \ldots & \vct{0}
\\
\vct{0} & \vct{1}_{l_2} & \ldots & \vct{0}
\\
\vdots & \vdots & \ddots & \vdots
\\
\vct{0} & \vct{0}& \ldots & \vct{1}_{l_r}
\end{bmatrix} \in \Mnr.
\]
To establish the theorem, it suffices to show for any feasible solution $\mtx{Y}$ with $\mtx{Y} \neq \mtx{Y}^\ast$,
\[
\Delta(\mtx{Y})  := \langle \mtx{Y}^\ast - \mtx{Y}, \mtx{A} -\lambda \vct{d}\vct{d}^\top \rangle >0.
\]

For a matrix $\mtx{X} \in \reals^{n \times n}$, let $\mtx{X}_w \in \reals^{n \times n}$ denote the matrix $\mtx{X}$ restricted to
entries in $\bigcup_{1 \le k \le r} C_k \times C_{k}$, and $ \mtx{X}_b \in \reals^{n \times n}$ denote the matrix $\mtx{X}$ restricted to entries in $\bigcup_{1 \le k < \ell \le r} C_{k} \times C_{\ell}$. It yields the decomposition $\mtx{X}=\mtx{X}_w + \mtx{X}_b$. Moreover, for each fixed pair $1 \leq a, b \leq r$, the submatrix of $\mtx{X}$ with entries in $C_{k} \times C_{\ell}$ is denoted as $\mtx{X}_{(a, b)} \in \mathbb{R}^{\ell_a \times \ell_b}$.

Let $\epsilon= \frac{\delta}{10}.$ We propose to decompose $\Delta(\mtx{Y})$ as
\begin{align}
\Delta(\mtx{Y})=&\langle \mtx{A}_w -\lambda (\vct{d}\vct{d}^\top)_w - \epsilon \left(\vct{\theta}\vct{\theta}^\top\right)_w, \mtx{Y}^\ast - \mtx{Y} \rangle  \nonumber
\\
                  &+\langle \epsilon \left(\vct{\theta}\vct{\theta}^\top\right)_w, \Ystar - \mtx{Y} \rangle \nonumber
\\
                  &+ \langle \E \mtx{A}_b - \lambda \left(\vct{d}\vct{d}^\top\right)_b, \Ystar - \mtx{Y} \rangle \nonumber
\\
                  &+ \langle \mtx{A}_b -\E \mtx{A}_b, \mtx{Y}^\ast - \mtx{Y} \rangle.  \nonumber
\\
              =:  & S_1 + S_2 + S_3 + S_4.                  \label{eq:DeltaY}
\end{align}
Below we establish lower bounds for the terms $S_1$, $S_2$, $S_3$ and $S_4$ respectively.

\paragraph{Lower bound of $S_1\,$}~\\
We plan to construct an $n \times n$ diagonal matrix $\mtx{D}$, such that with high probability, 
\begin{equation}
\label{eq:dual1}
\begin{cases}
\mtx{\Lambda}:= \mtx{D} + \epsilon \left(\vct{\theta}\vct{\theta}^\top\right)_w + \lambda \left(\vct{d}\vct{d}^\top\right)_w - \mtx{A}_w \succeq \mtx{0} ;
\\
\mtx{\Lambda}\mtx{\Psi}^\ast= \mtx{0}.
\end{cases}
\end{equation}
Such a diagonal matrix $\mtx{D}$ implies that with high probability, 
\begin{align*}
S_1&=\langle \mtx{A}_w -\lambda (\vct{d}\vct{d}^\top)_w - \epsilon \left(\vct{\theta}\vct{\theta}^\top\right)_w, \mtx{Y}^\ast - \mtx{Y} \rangle
\\
&  \overset{(a)}{=} \langle \mtx{A}_w -\lambda (\vct{d}\vct{d}^\top)_w - \epsilon \left(\vct{\theta}\vct{\theta}^\top\right)_w - \mtx{D}, \mtx{Y}^\ast - \mtx{Y} \rangle
\\
& \overset{(b)}{=} \langle -\mtx{\Lambda}, \mtx{Y}^\ast - \mtx{Y} \rangle \overset{(c)}{\geq} \langle -\mtx{\Lambda}, \mtx{Y}^\ast \rangle = 0,
\end{align*}
where the step $(a)$ follows $\diag(\Ystar)=\diag(\mtx{Y})=\mtx{I}_n$, $(b)$ follows from the definition of $\mtx{\Lambda}$,
$(c)$ holds due to $\mtx{Y} \succeq \mtx{0}$ and $\mtx{\Lambda} \succeq \mtx{0}$, and the last equality follows because $\mtx{\Lambda}\mtx{\Psi}^\ast= \mtx{0}$.

In what follows, we will show how to construct explicitly the diagonal matrix $\mtx{D}$ satisfying the condition~\eqref{eq:dual1}
with high probability. Notice that the condition~\eqref{eq:dual1} is equivalent to  that with high probability, for all $1 \leq a \leq r$,
\begin{equation}
\label{eq:dual2}
\begin{cases}
\mtx{\Lambda}_{(a, a)}:= \mtx{D}_{(a, a)} + \epsilon \vct{\theta}_{(a)} \vct{\theta}_{(a)}^\top + \lambda \vct{d}_{(a)}\vct{d}_{(a)}^\top - \mtx{A}_{(a, a)} \succeq \mtx{0};
\\
\mtx{\Lambda}_{(a, a)}\vct{1}_{\ell_a}= \vct{0},
\end{cases}
\end{equation}
where $\vct{1}_{\ell}$ denote the $\ell$-dimensional vector whose coordinates all equal $1$. 
The equality condition gives
\begin{equation}
\label{def:D}
\mtx{D}_{(a, a)}=\diag\left(\mtx{A}_{(a, a)}\vct{1}_{\ell_a} - \lambda\left(\sum_{j \in C^\ast_a} d_j\right)\vct{d}_{(a)} - \epsilon\; G_a \vct{\theta}_{(a)}\right).
\end{equation}
The equality condition also implies that $\rank \left(\mtx{\Lambda}_{(a, a)}\right)\leq \ell_a - 1$. Therefore, in order to prove $\mtx{\Lambda}_{(a, a)} \succeq \mtx{0}$, it suffices to prove $ \lambda_{\ell_a - 1} \left(\mtx{\Lambda}_{(a, a)}\right)>0$. By Weyl's inequality~\citep{HJ2013}, we have
\begin{align*}
 & \lambda_{\ell_a - 1}\left(\mtx{\Lambda}_{(a,a)}\right) \\
 =&  \lambda_{\ell_a - 1}\left(\mtx{D}_{(a, a)} - \mtx{A}_{(a, a)} +  \lambda \vct{d}_{(a)} \vct{d}_{(a)}^\top + \epsilon \vct{\theta}_{(a)}\vct{\theta}_{(a)}^\top\right)
 \\
 \geq& \lambda_{\ell_a}\left(\mtx{D}_{(a, a)} - \mtx{A}_{(a, a)} + p \vct{\theta}_{(a)}\vct{\theta}_{(a)}^\top \right) + \lambda_{\ell_a - 1}\left( \lambda \vct{d}_{(a)} \vct{d}_{(a)}^\top + \epsilon \vct{\theta}_{(a)}\vct{\theta}_{(a)}^\top-p \vct{\theta}_{(a)}\vct{\theta}_{(a)}^\top\right)
 \\
 \geq& \lambda_{\ell_a}\left(\mtx{D}_{(a, a)} - \mtx{A}_{(a, a)} + p \vct{\theta}_{(a)}\vct{\theta}_{(a)}^\top \right).
 \end{align*}
The last inequality is due to the fact that $\lambda \vct{d}_{(a)} \vct{d}_{(a)}^\top + \epsilon \vct{\theta}_{(a)}\vct{\theta}_{(a)}^\top-p \vct{\theta}_{(a)}\vct{\theta}_{(a)}^\top$ has at most one negative eigenvalue. Therefore, to establish \eqref{eq:dual1}, we only need to prove
\[
\mtx{D}_{(a, a)} - \mtx{A}_{(a, a)} + p \vct{\theta}_{(a)}\vct{\theta}_{(a)}^\top \succ \mtx{0},
\]
or equivalently,
\begin{equation*}
\mtx{\Lambda}_{1}:=\diag(\vct{\theta}_{(a)})^{-\frac{1}{2}}
\left(\mtx{D}_{(a, a)} - \mtx{A}_{(a, a)} + p \vct{\theta}_{(a)}\vct{\theta}_{(a)}^\top \right)\diag(\vct{\theta}_{(a)})^{-\frac{1}{2}} \succ \mtx{0}.
\end{equation*}
Define the matrices
\[
\begin{cases}
\mtx{\Lambda}_{11}:=\diag(\vct{\theta}_{(a)})^{-\frac{1}{2}}\mtx{D}_{(a, a)}\diag(\vct{\theta}_{(a)})^{-\frac{1}{2}} = \mtx{D}_{(a, a)}\diag\left(\vct{\theta}_{(a)}\right)^{-1},
\\
\mtx{\Lambda}_{12}:= \diag(\vct{\theta}_{(a)})^{-\frac{1}{2}} \left(\mtx{A}_{(a, a)} - p \vct{\theta}_{(a)}\vct{\theta}_{(a)}^\top)\right) \diag(\vct{\theta}_{(a)})^{-\frac{1}{2}}.
\end{cases}
\]
Then $\mtx{\Lambda}_{1} = \mtx{\Lambda}_{11} + \mtx{\Lambda}_{12}$. By Weyl's inequality~\citep{HJ2013}, to prove $\mtx{\Lambda}_{1} \succ \mtx{0}$, we only need to show that
\begin{equation}
\label{eq:positivity}
\lambda_{\ell_a}(\mtx{\Lambda}_{11}) > \|\mtx{\Lambda}_{12}\|.
\end{equation}
Applying Lemma~\ref{teo:wigner_spec} in the appendix, we can prove that with probability at least $1 - \frac{1}{n^2}$,
\begin{equation}
\label{eq:norm_upper}
\left\| \mtx{\Lambda}_{12}\right\|\leq C_1\left(\sqrt{p \ell_a \log n} + \frac{\log n}{\theta_{\min}}\right) + p,
\end{equation}
for some numerical constant $C_1$.  Moreover,
\begin{equation}
\label{eq:diagonal_min}
\lambda_{\ell_a}(\mtx{\Lambda}_{11})  = \min_{i \in C^\ast_a} \left(\frac{1}{\theta_i}\left(\sum_{j \in C^\ast_a} a_{ij} \right) - \lambda \left(\sum_{j \in C^\ast_a} d_{j} \right) \frac{d_i}{\theta_i} - \epsilon G_a\right),
\end{equation}
where  $a_{ij}$ denotes the $(i,j)$-th entry of $\mtx{A}.$
By Chernoff's inequality, for each $i \in C^\ast_a$, with probability at least $1 - \frac{1}{n^3}$,
\begin{align*}
\sum_{j \in C^\ast_a} a_{ij} &\geq p\theta_i\left(\sum_{j \in C^\ast_a/\{i\}} \theta_j \right) - \sqrt{6 (\log n) \theta_i p \left(\sum_{j \in C^\ast_a/\{i\}} \theta_j \right)}
\\
&\geq p\theta_i\left(G_a -1 \right) - \sqrt{6 (\log n) \theta_i p G_a}.
\end{align*}
By the bound \eqref{eq:d_bound2} and the separation condition \eqref{eq:sep_perf}, with probability at least $1 - \frac{2}{n^2}$ there holds the bound
\begin{align*}
\lambda \left(\sum_{j \in C^\ast_a} d_j\right) \frac{d_i}{\theta_i} &\leq \frac{p - \delta}{H_a^2 \theta_i} \left(1 + \frac{\delta}{5p}\right)^2\left(\sum_{j \in C^\ast_a}f_j\right)f_i
\\
& = \frac{p - \delta}{H_a^2 \theta_i} \left(1 + \frac{\delta}{5p}\right)^2\left(\sum_{j \in C^\ast_a}\theta_j H_a \right)\theta_i H_a
\\
& = (p - \delta) \left(1 + \frac{\delta}{5p}\right)^2 G_a \leq \left(p - \frac{3}{5}\delta \right) G_a.
\end{align*}
Then by the fact $\epsilon = \frac{\delta}{10}$, the expression~\eqref{eq:diagonal_min} and the union bound, we conclude that with probability at least $1-\frac{3}{n^2}$,
\[
\lambda_{\ell_a}(\mtx{\Lambda}_{11}) \geq \frac{1}{2} \delta G_a - p - \sqrt{\frac{6(\log n)pG_a}{\theta_{\min}}}.
\]
Therefore, to guarantee \eqref{eq:positivity}, it suffices to let
\[
\frac{1}{2} \delta G_a - p - \sqrt{\frac{6(\log n)pG_a}{\theta_{\min}}} > C_1\left(\sqrt{p \ell_i \log n} + \frac{\log n}{\theta_{\min}}\right) + p.
\]
Since $\ell_i \le G_i/\theta_{\min}$ and $p\leq 1$, the above inequality is guaranteed by the assumption  \eqref{eq:assump} and its implication \eqref{eq:lower_sim} with sufficiently large $C_0$. Thus for each $1 \le a \le r$,  \eqref{eq:dual2} holds with probability at least $1 - \frac{4}{n^2}$. By the union bound,  \eqref{eq:dual2} holds for all $1 \le a \le r$ with probability at least $1-\frac{4}{n}$. 

\paragraph{Lower bound of $S_2\,$}~\\
For any $(i, j) \in C^\ast_a \times C^\ast_a$ for some $a=1, \ldots, r$, $Y^*_{ij}=1 \geq Y_{ij}$. This implies that the matrix $\mtx{Y}_w^* - \mtx{Y}_w$ is entrywise nonnegative, and thus
\[
S_2 = \langle \epsilon \left(\vct{\theta}\vct{\theta}^\top\right)_w, \Ystar - \mtx{Y} \rangle = \frac{\delta}{10} \|\mtx{Y}_w^* - \mtx{Y}_w\|_{1, \vct{\theta}}
\label{eq:lowerboundS2}.
\]

\paragraph{Lower bound of $S_3\,$}~\\
For any $(i, j) \in C^\ast_a \times C^\ast_b$ with $a \neq b$, by the bound~\eqref{eq:d_bound2}, with probability at least $1 - \frac{2}{n^2}$ there holds
\[
d_i d_j \geq \left(1-\frac{\delta}{5p}\right)^2 f_i f_j = \left(1-\frac{\delta}{5p}\right)^2 \theta_i \theta_j H_aH_b.
\]
The separation condition \eqref{eq:sep_perf} implies that $\lambda > \frac{q+\delta}{H_aH_b}$, so with probability at least $1 - \frac{2}{n^2}$,
\[
\lambda d_i d_j > (q+\delta)\left(1-\frac{\delta}{5p}\right)^2 \theta_i \theta_j.
\]
The inequalities
\[
(q+\delta)\left(1-\frac{\delta}{5p}\right)^2 > q + \delta - \frac{2q\delta}{5p} - \frac{2\delta^2}{5p}>q+\delta - \frac{4\delta}{5}=q+\frac{1}{5}\delta
\]
imply that with probability at least $1 - \frac{2}{n^2}$,
\[
\lambda d_i d_j > \left(q + \frac{1}{5}\delta  \right) \theta_i \theta_j.
\]
Moreover, we know for any $(i, j) \in C^\ast_a \times C^\ast_b$ with $a \neq b$, $Y^*_{ij}=0$. Therefore, with probability at least $1 - \frac{2}{n^2}$,
\begin{align}
S_3 &= \langle \E \mtx{A}_b - \lambda \left(\vct{d}\vct{d}^\top\right)_b, \Ystar - \mtx{Y} \rangle =\Iprod{ q(\vct{\theta}\vct{\theta}^\top)_b  - \lambda (\vct{d}\vct{d}^\top)_b}{-\mtx{Y}_b}\geq \frac{\delta}{5} \|\mtx{Y}_b\|_{1, \vct{\theta}}.
 \label{eq:boundmean}
\end{align}

\paragraph{Lower bound of $S_4\,$}~\\
Define matrix $\mtx{W} = ( \mtx{A}_b -\expect{\mtx{A}_b}) \circ \mtx{\Theta}^{-\frac{1}{2}}$ and then
\begin{align}
     \langle \mtx{A}_b -\expect{\mtx{A}_b}, \mtx{Y}^\ast- \mtx{Y} \rangle = \langle \mtx{W}, (\mtx{Y}^\ast-  \mtx{Y }) \circ \mtx{\Theta}^{\frac{1}{2}} \rangle.
 \end{align}

Let $ \mtx{U} \in \reals^{n \times r}$ be the weighted characteristic matrix for the clusters, \ie,
$$\mtx{U}_{ia} = \begin{cases}
\frac{\sqrt{\theta_i} }{ \sqrt{ \|\theta^{(a)}\|_1 }} & \text{if node $i \in C^\ast_a$} \\
0 & \text{otherwise,}
\end{cases}$$
Let $ \mtx{\Sigma} \in \reals^{r \times r }$ be the diagonal matrix with ${\Sigma}_{aa}=  \|\theta^{(a)}\|_1$ for $a \in [r]$.
The weighted true cluster matrix $ \Ystar \circ \mtx{\Theta}^{\frac{1}{2}} $ has a rank-$ r $ singular value
decomposition given by $\mtx{Y}^{\ast} \circ \mtx{\Theta}^{\frac{1}{2}}= \mtx{U} \mtx{\Sigma} \mtx{U}^{\top}$.
Define the projections $\mathcal{P}_T( \mtx{M} ) =
\mtx{U} \mtx{U}^\top \mtx{M}+\mtx{M}\mtx{U}\mtx{U}^\top-\mtx{U}\mtx{U}^\top \mtx{M}\mtx{U}\mtx{U}^\top $ and $\mathcal{P}_{T^\perp}(\mtx{M}) = \mtx{M} - \mathcal{P}_T(\mtx{M})$.
Let $\Tr(\mtx{M})$ denote the trace of $\mtx{M}$.

It follows from \prettyref{lmm:matrixconcentration} in the Appendix that with probability at least $1 - \frac{1}{n^2}$,
\begin{align*}
\| \mtx{W} \| \leq c_2 \sqrt{  qn   } + c_2 \frac{ \sqrt{ \log n} }{\theta_{\min}}.
\end{align*}

Notice that $ \mtx{U} \mtx{U}^\top+ \mathcal{P}_{T^\perp}\left(\frac{ \mtx{W} }{\| \mtx{W} \|} \right)$ is a subgradient of $\| \mtx{X} \|_*$ at
$\mtx{X} = \mtx{Y}^\ast \circ \mtx{\Theta}^{\frac{1}{2}}$.
Hence, we have
 \begin{align*}
   0 &\ge \Tr \left( \mtx{Y} \circ \mtx{\Theta}^{\frac{1}{2}} \right)-\Tr \left( \mtx{Y}^\ast \circ \mtx{\Theta}^{\frac{1}{2}} \right) \\
   &\geq  \langle \mtx{U}\mtx{U}^\top+ \mathcal{P}_{T^\perp}\left( \mtx{W} / \|\mtx{W}\| \right), (\mtx{Y}- \mtx{Y}^\ast) \circ \mtx{\Theta}^{\frac{1}{2}} \rangle,
   \end{align*}
   and as a consequence,
   \begin{align*}
   \langle  \mathcal{P}_{T^\perp}\left( \mtx{W} \right), (\mtx{Y}^\ast- \mtx{Y}) \circ \mtx{\Theta}^{\frac{1}{2}} \rangle \ge \|\mtx{W} \| \Iprod{\mtx{U} \mtx{U}^\top}{(\mtx{Y}- \mtx{Y}^\ast) \circ \mtx{\Theta}^{\frac{1}{2}} }
   \end{align*}
   Therefore, we get that
   \begin{align}
  &  \langle \mtx{W}, (\mtx{Y}^\ast-  \mtx{Y }) \circ \mtx{\Theta}^{\frac{1}{2}} \rangle \nonumber \\
    =&  \langle  \mathcal{P}_{T^\perp} (\mtx{W} ), (\mtx{Y}^\ast-  \mtx{Y }) \circ \mtx{\Theta}^{\frac{1}{2}} \rangle + \langle  \mathcal{P}_{T} (\mtx{W} ), (\mtx{Y}^\ast-  \mtx{Y }) \circ \mtx{\Theta}^{\frac{1}{2}} \rangle \nonumber \\
    =&  \langle  \mathcal{P}_{T^\perp} (\mtx{W} ), (\mtx{Y}^\ast-  \mtx{Y }) \circ \mtx{\Theta}^{\frac{1}{2}} \rangle - \langle  \mathcal{P}_{T} (\mtx{W} ), \mtx{Y }_b \circ \mtx{\Theta}^{\frac{1}{2}} \rangle \nonumber \\
    \ge& \|\mtx{W} \|   \Iprod{\mtx{U} \mtx{U}^\top}{(\mtx{Y}- \mtx{Y}^\ast) \circ \mtx{\Theta}^{\frac{1}{2}} }
    - \langle \mathcal{P}_{T}( \mtx{W} ) , \mtx{Y}_b \circ \mtx{\Theta}^{\frac{1}{2}} \rangle  \nonumber \\
   \geq&  -  \| \mtx W\| \| \mtx{U} \mtx{U}^\top\|_{\infty, \mtx{ \Theta}^{-\frac{1}{2}}}   \| \mtx{Y}^\ast - \mtx{Y} \|_{1, \vct{\theta}}  - \|\mathcal{P}_{T}( \mtx{ W} )\|_{\infty,  \mtx{\Theta}^{-\frac{1}{2}}} \| \mtx Y_b \|_{1, \vct{\theta}} ,
   \label{eq:DeltaYLowerBound}
  \end{align}
  where the second equality holds because $\mathcal{P}_{T} (\mtx{W} ) =0$ on the diagonal blocks $C_k \times C_k$ for $1 \le k \le r$; the last inequality follows because for any matrix $\mtx{M} \in \reals^{n \times n}$, $\| \mtx{M} \|_{\infty,  \mtx{\Theta}^{-\frac{1}{2}}} := \| \mtx{M} \circ \mtx{\Theta}^{-1/2} \|_{\infty}$. 
  By the definition of $\mtx U$, $\| \mtx{U} \mtx{U}^\top\|_{\infty, \mtx{\Theta}^{-\frac{1}{2}}}= 1/ G_{\min}$.  It follows from
  the theorem assumption \prettyref{eq:assump} that
$
  \delta > \frac{10 \| \mtx{W} \| }{G_{\min} }.
$
Therefore, we have
  \begin{align}
  \| \mtx W\| \| \mtx{U} \mtx{U}^\top\|_{\infty,  \mtx{\Theta}^{-\frac{1}{2}}}    =   \frac{\| \mtx{W} \| }{G_{\min} }     < \frac{\delta }{10} . \label{eq:S4first}
  \end{align}

   Below we bound the term $\|\mathcal{P}_{T}(\mtx{W} )\|_{\infty, \Theta^{-\frac{1}{2}}}$. From the definition of $\mathcal{P}_T$,
  \begin{align*}
  & \|\mathcal{P}_T(\mtx W)\|_{\infty, \Theta^{-\frac{1}{2}}} \\
  \leq&  \| \mtx{U} \mtx{U} ^\top \mtx{W} \|_{\infty, \mtx{\Theta}^{-\frac{1}{2}}}+\|\mtx{W} \mtx{U} \mtx{U}^\top \|_{\infty, \mtx{\Theta}^{-\frac{1}{2}}}
    + \|\mtx{U} \mtx{U}^\top \mtx{W} \mtx{U} \mtx{U} ^\top \|_{\infty, \mtx{\Theta}^{-\frac{1}{2}}} . 
  \end{align*}
  We bound $\| \mtx{U} \mtx{U} ^\top \mtx{W} \|_{\infty, \mtx{\Theta}^{-\frac{1}{2}}}$ below. Notice that $( \mtx{U} \mtx{U} ^\top \mtx{W} )_{ij} =0$ if $i$ and $j$ are from the same cluster.
  Thus, to bound the term $( \mtx{U} \mtx{U}^\top \mtx{W} )_{ij}$, it suffices to consider the case where
   $i$ belongs to cluster $k$ and $j$ belongs to a different cluster $k'$ for $k' \neq k$, Recall $C^\ast_k$ is the set of users in cluster $k$.
  Then
  \begin{align*}
    (\mtx{U} \mtx{U}^\top \mtx{W} )_{ij}  = \frac{\sqrt{ \vct{\theta}_i}}{\| \vct{\theta}^{(a)}\|_1 }  \sum_{i' \in C^\ast_k } \sqrt{ {\theta}_{i'}} {W}_{i'j},
  \end{align*}
  which is the weighted average of independent random variables.
   By Bernstein's inequality, with probability at least $1-n^{-3}$,
  \begin{align*}
    \left|\sum_{i'\in C^\ast_k} \sqrt{ {\theta}_{i'}} {W} _{i'j} \right|\leq  \sqrt{6 q \| \vct{\theta}^{(a)}\|_1 \log n } + \frac{2 \log n}{\sqrt{ {\theta}_{j} }}. 
  \end{align*}
  Then  with probability at least $1-n^{-1}$,
  \begin{align*}
  \| \mtx{U} \mtx{U}^\top \mtx W\|_{\infty,  \mtx{\Theta}^{-\frac{1}{2}}}
   \leq c_1  \left(  \sqrt{ \frac{ q \log n }{ G_{\min} \theta_{\min}} } + \frac{\log n}{G_{\min} \theta_{\min} } \right).
   \end{align*}
    Similarly we bound $\| \mtx{W} \mtx{U} \mtx{U}^\top \|_{\infty, \mtx{\Theta}^{-\frac{1}{2}}}$ and $\| \mtx{U} \mtx{U}^\top  \mtx{W} \mtx{U} \mtx{U}^\top \|_{\infty, \mtx{\Theta}^{-\frac{1}{2}}}$.
    Therefore, with probability at least $1-3n^{-1}$,
  \begin{align}
    \|\mathcal{P}_T(\mtx{W})\|_{\infty, \mtx{\Theta}^{-\frac{1}{2}}} \le 3 c_1  \left(  \sqrt{ \frac{ q \log n }{ G_{\min} \theta_{\min}} } + \frac{\log n}{G_{\min} \theta_{\min} } \right)
    < \delta/10,
     \label{eq:PTW}
  \end{align}
  where the last inequality follows from the theorem assumption \prettyref{eq:assump}.
Substituting  the bounds~\prettyref{eq:S4first} and \prettyref{eq:PTW} into the inequality~\prettyref{eq:DeltaYLowerBound}, we get that with  probability at least $1-4n^{-1}$,
\begin{align*}
S_4 > - \frac{\delta}{10} \left(  \| \mtx{Y}^\ast - \mtx{Y} \|_{1, \vct{\theta} }  +  \|\mtx{Y}_b \|_{1, \vct{\theta} }   \right).
\end{align*}

\paragraph{Putting together}~\\
Combining the bounds for $ S_i$ with $i=1,2,3,4 $, we conclude that with  probability at least $1-10n^{-1}$,
\begin{align*}
  & S_1+S_2+S_3+S_4 \\
>&   \frac{\delta}{10} \|\mtx{Y}_w^* - \mtx{Y}_w\|_{1, \vct{\theta}}+ \frac{\delta}{5}  \| \mtx{Y}_b  \|_{1, \vct{\theta}}  - \frac{\delta}{10} \left(  \| \mtx{Y}^\ast - \mtx{Y} \|_{1, \vct{\theta} }  +  \|\mtx{Y}_b \|_{1, \vct{\theta} }   \right) \\
\geq& 0,
\end{align*}
thereby proving that $\Delta(Y)>0$ for any feasible $\mtx{Y} \neq \mtx{Y}^\ast$.

\end{proof}

\subsubsection{Proof of \prettyref{lmm:degrees}}
\label{sec:proof_lmm_degrees}

\begin{proof}
The equation \eqref{eq:d_bound1} can be obtained straightforwardly by Chernoff's inequality. To prove \eqref{eq:d_bound2}, we only need to establish for all $i=1, \ldots, n$,
\begin{equation}
\label{eq:delta_lower_1}
\delta \geq 5p \max \left(\sqrt{\frac{12 \log n}{f_i}}, \frac{4 \log n + 1}{f_i}\right).
\end{equation}
For any $i \in C^\ast_a$, since
\[
f_i = \theta_i H_a \geq  p \theta_i  G_{\min} \geq p \theta_{\min}  G_{\min},
\]
Therefore,  the assumption~\eqref{eq:assump} implies that
\[
\delta > C_0 \sqrt{\frac{p \log n}{G_{\min}\theta_{\min}}} = C_0p \sqrt{\frac{\log n}{pG_{\min}\theta_{\min}}} \geq C_0p \sqrt{\frac{\log n}{f_i}}.
\]
Therefore, as long as $C_0$ is large enough, we have
\begin{equation}
\label{eq:delta_lower_2}
\delta \geq 5p \sqrt{\frac{12 \log n}{f_i}}.
\end{equation}
Since $\delta < p$, for sufficiently large $C_0$, we have that $\sqrt{\frac{12 \log n}{f_i}} \leq \frac{1}{5}$, and this implies
\[
\frac{4 \log n + 1}{f_i} \leq \frac{12 \log n}{f_i} \leq \sqrt{\frac{12 \log n}{f_i}} .
\]
The bound \eqref{eq:delta_lower_1} can then be deduced from~\eqref{eq:delta_lower_2}.
\end{proof}


\section*{Acknowledgment}

Y.\ Chen was supported by the School of Operations Research and Information Engineering at Cornell University. 
X.\ Li was supported by a startup fund from the Statistics Department at University of California, Davis.
J.\ Xu was supported by the National Science Foundation under
Grant CCF 14-09106, IIS-1447879, OIS 13-39388, and CCF 14-23088, and
Strategic Research
Initiative on Big-Data Analytics of the College of Engineering
at the University of Illinois, DOD ONR Grant N00014-14-1-0823, and Simons Foundation Grant 328025.

\appendixpage
\appendix
\section{Supporting lemmas}
In this section we state several additional technical lemmas concerning random matrices. These lemmas are used in the proof of our main theorems.

Recall that  $\circ$ denotes the element-wise product between matrices.

\begin{lemma}\label{lmm:matrixconcentration}
Let $\mtx{W} = ( \mtx{A}_b -\expect{\mtx{A}_b}) \circ \mtx{\Theta}^{-\frac{1}{2}}$.
 For any $c>0$,
there exists $c'>0$, such that with probability at least $1-n^{-c}$,
\begin{align}
\|  \mtx{W}  \| \le c' \left( \sqrt{n q} + \frac{ \sqrt{ \log n} }{ \theta_{\min} } \right).
\end{align}
\end{lemma}
\begin{proof}
Let $\mtx{W}'$ denote an independent copy of $\mtx{W}$. Let $\mtx{M}=(\mtx{M}_{ij})$ denote an $n\times n$ zero-diagonal symmetric matrix whose entries are Rademacher and independent
from $\mtx{W}$ and $\mtx{W}'$. We apply the usual symmetrization arguments:
\begin{align}
\Expect[\| \mtx{W}  \|]
= & ~ \Expect[\|\mtx{W}  - \Expect[\mtx{W} ' ] \|] \nonumber \\
\overset{(a)}{\leq} &~ \Expect[\| \mtx{W}-\mtx{W}' \|]  \nonumber \\
\overset{(b)}{=} &~  \Expect[\| ( \mtx{W}-\mtx{W}' ) \circ M \| ] \nonumber \\
\overset{(c)}{\le} &~ 2\Expect[ \| \mtx{W}  \circ \mtx{M} \|],
\nonumber 
\end{align}
where $(a)$ follow from  the Jensen's inequality, $(b)$ follows because $\mtx{W}-\mtx{W}'$ has the same distribution as $(\mtx{W}-\mtx{W}')\circ \mtx{M} $, and $(c)$ follow from the triangle inequality.

Next we upper bound $\Expect[ \| \mtx{W}  \circ \mtx{M}\|].$ Notice that $\mtx{W}  \circ \mtx{M}$ is a $n \times n$ symmetric and zero-diagonal
random matrix, where the entries $\{ ( \mtx{W}  \circ \mtx{M} )_{ij}, i<j \}$ are independent. Let $b_{ij} = \sqrt{q (1-q \theta_i \theta_j ) } $ if $i$ and $j$ in two different clusters; otherwise $b_{ij}=0$.
Let $\{ \xi_{ij}, i \le j \}$ denote independent random variables with
\begin{align*}
\xi_{ij} = \left \{
    \begin{array}{rl}
    \frac{1-q \theta_{i} \theta_{j} }{\sqrt{q (1-q \theta_i \theta_j ) \theta_i \theta_j} }  & \text{w.p. } \frac{1}{2} q \theta_i \theta_j  \\
    -   \frac{1-q \theta_{i} \theta_{j} }{\sqrt{ q (1-q \theta_i \theta_j ) \theta_i \theta_j} }  & \text{w.p. }  \frac{1}{2} q \theta_i \theta_j \\
    \frac{q \sqrt{ \theta_i \theta_j } }{\sqrt{q (1-q \theta_i \theta_j )} } & \text{w.p. }  \frac{1}{2} (1- q \theta_i \theta_j ) \\
     -\frac{q \sqrt{ \theta_i \theta_j } }{\sqrt{q (1-q \theta_i \theta_j )} } & \text{w.p. }  \frac{1}{2} (1- q \theta_i \theta_j ).
    \end{array} \right.
    \end{align*}
    Notice that $\xi_{ij}$ has a symmetric distribution with zero mean and unit variance.  Let $\mtx{X}$ denote the random matrix with $\mtx{X}_{ij} = \mtx{X}_{ij} =\xi_{ij} b_{ij}.$
    Then one can check that $\mtx{W}  \circ \mtx{M}$ has the same distribution as $\mtx{X}$. Notice that $\|\mtx{X} \|_{\infty} \le 1/\theta_{\min}$. It follows from \cite[Corollary 3.6]{BVH14}
    that there exists some absolute constant $c_1>0$ such that 
    \begin{align*}
    \expect{\mtx{W}  \circ \mtx{M} } = \expect{\mtx{X} } \le c_1 \left( \sqrt{nq} + \frac{\sqrt{ \log n} }{\theta_{\min} } \right).
    \end{align*}
Since the entries of $\| \mtx{W} \|_{\infty} \le 1/\theta_{\min} $, Talagrand's concentration inequality for 1-Lipschitz convex functions (see, \eg, \cite[Theorem 2.1.13]{tao.rmt}) yields the bound
\[
\prob{\| \mtx{W} ]\| \geq \Expect[\| \mtx{W}  \|]+t/ \theta_{\min} } \leq c_2 \exp(-c_3 t^2)
\]
for some absolute constants $c_2,c_3$, which implies that for any $c>0$,
there exists $c'>0$, such that
\begin{align*}
\prob{\| \mtx{W} \| \geq  c' \left( \sqrt{nq} + \frac{\sqrt{ \log n} }{\theta_{\min} } \right)  } \leq n^{-c}.
\end{align*}

\end{proof}


\begin{lemma}
[Theorem 6.1 in \cite{Tropp2011}]
\label{teo:tropp}
Consider a finite sequence $\{\mtx{X}_k\}$ of independent, random, self-adjoint matrices with dimension $d$. Assume that
\[
\E \mtx{X}_k = \mtx{0} \quad \text{and} \quad \|\mtx{X}_k\|\leq R.
\]
If the norm of the total variance satisfies
\[
\left\| \sum_{k}\E (\mtx{X}_k^2)\right\|\leq M^2,
\]
then, the following inequality holds for all $t\geq 0$
\[
\P\left\{\left\|\sum_{k}\mtx{X}_k\right\|\geq t\right\}\leq 2d \exp\left(\frac{-t^2/2}{M^2+Rt/3}\right).
\]
\end{lemma}

\begin{lemma}
\label{teo:wigner_spec}
Let $\mtx{A}=(A_{ij})_{1\leq i, j\leq n}$ be a symmetric random matrix whose diagonal entries are all zeros. Moreover, suppose $A_{ij}$, $1\leq i<j\leq n$ are independent zero-mean random variables satisfying $|A_{ij}|\leq R$ and $\Var(A_{ij})\leq \sigma^2$. Then, with probability at least $1-\frac{2}{n^4}$, we have
\[
\|\mtx{A}\|\leq C_0 \left(\sigma \sqrt{n \log n}+ R\log n \right)
\]
for some numerical constant $C_0$.
\end{lemma}

\begin{proof}
For each pair $(i, j): 1\leq i<j\leq n$, let $\mtx{X}_{ij}$ be the matrix whose $(i, j)$ and $(j, i)$ entries are both $A_{ij}$, whereas other entires are zeros. Then we have
\[
\mtx{A}=\sum_{1\leq i<j \leq n} \mtx{X}_{ij}.
\]
Moreover, we can easily show that $\E \mtx{X}_{ij}=\mtx{0}$, $\|\mtx{X}_{ij}\|\leq R$ and
\[
\mtx{0}\preceq \sum_{1\leq i<j\leq n} \E\mtx{X}_{ij}^2 \preceq (n-1)\sigma^2 \mtx{I}_n.
\]
Applying Lemma \ref{teo:tropp} completes the proof.
\end{proof}

\bibliographystyle{plainnat_ini}
\bibliography{DCSBM_CMM}


\end{document}